\documentclass[10pt,a4paper]{amsart}
\usepackage{amsfonts,amsmath,amssymb}
\usepackage{hyperref}
\usepackage[all]{xy}
\newtheorem*{theorem*}{Theorem}
\newtheorem{lemma}{Lemma}[subsection]
\newtheorem{proposition}[lemma]{Proposition}
\newtheorem{remark}[lemma]{Remark}
\newtheorem{example}[lemma]{Example}
\newtheorem{theorem}[lemma]{Theorem}
\newtheorem{definition}[lemma]{Definition}
\newtheorem{notation}[lemma]{Notation}
\newtheorem{property}[lemma]{Property}
\newtheorem{corollary}[lemma]{Corollary}

\newtheorem{conjecture}{Conjecture}
\newtheorem*{conjecture*}{Conjecture}
\baselineskip 18pt \textwidth 16cm  \theoremstyle{plain}

\newcommand{\tr}{\operatorname{Tr}}

\newcommand{\ad}{\operatorname{ad}}

\newcommand{\Hom}{\operatorname{Hom}}

\newcommand{\cc}{\mathbb{C}}

\newcommand{\eps}{\varepsilon}

\newcommand{\re}{\operatorname{Re}}
\renewcommand{\Im}{\operatorname{Im}}

\newcommand{\Z}{{\mathbb Z}}

\newcommand{\R}{{\mathbb R}}
\newcommand{\C}{{\mathbb C}}

\newcommand{\E}{{\mathcal E}}

\newcommand{\Aut}{{\operatorname{Aut}}}

\newcommand{\End}{\operatorname{End}}

\newcommand{\G}{{\mathcal G}}

\newcommand{\Fre}{{Fr\'{e}chet \,}}
\newcommand{\et}{{\'{e}tale }}

\newcommand{\Fou}{{\mathcal{F}}}

\newcommand{\cD}{{\mathcal{D}}}
\newcommand{\g}{{\mathfrak{g}}}
\newcommand{\h}{{\mathfrak{h}}}
\newcommand{\z}{{\mathfrak{z}}}

\newcommand{\Supp}{\mathrm{Supp}}

\newcommand{\gd}{\g^{\sigma}}
\newcommand{\oF}{{\overline{F}}}
\newcommand{\cO}{{\mathcal{O}}}
\newcommand{\GL}{\operatorname{GL}}
\newcommand{\SL}{\operatorname{SL}}
\newcommand{\gl}{{\mathfrak{gl}}}
\newcommand{\sll}{{\mathfrak{sl}}}
\newcommand{\HC}{\operatorname{HC}}
\newcommand{\Fr}{\operatorname{Fr}}
\newcommand{\Mp}{\operatorname{Mp}}
\newcommand{\Sym}{\operatorname{Sym}}
\newcommand{\Spec}{\operatorname{Spec}}
\newcommand{\Sc}{{\mathcal S}}
\newcommand{\Lie}{\operatorname{Lie}}
\newcommand{\Ad}{\operatorname{Ad}}
\usepackage[usenames]{color}
\newcommand{\Rami}[1]{{{#1}}}

\begin{document}

\author{Avraham Aizenbud}
\address{Avraham Aizenbud and Dmitry Gourevitch, Faculty of Mathematics
and Computer Science, The Weizmann Institute of Science POB 26,
Rehovot 76100, ISRAEL.} \email{aizenr@gmail.com}
\author{Dmitry Gourevitch} \email{dimagur@gmail.com}

\title[Generalized Harish-Chandra descent]{Generalized Harish-Chandra descent, Gelfand pairs and an Archimedean analog of Jacquet-Rallis' Theorem}
%
%
%
%
%
%
%
%
%
%
\maketitle
$\quad \quad \quad$ with Appendix \ref{SecRedLocPrin} by Avraham
Aizenbud, Dmitry Gourevitch and Eitan Sayag

\begin{abstract}
In the first part of the paper we generalize a descent technique
due to Harish-Chandra to the case of a reductive group acting on a
smooth affine variety both defined over an arbitrary local field
$F$ of characteristic zero. Our main tool is the Luna Slice
Theorem.

In the second part of the paper we apply this technique to
symmetric pairs. In particular we prove that the pairs
$(\GL_{n+k}(F),\GL_n(F) \times \GL_k(F))$ and
$(\GL_n(E),\GL_n(F))$ are Gelfand pairs for any local field $F$
and its quadratic extension $E$. In the non-Archimedean case, the
first result was proven earlier  by Jacquet and Rallis and the
second by Flicker.

We also prove that any conjugation invariant distribution on
$\GL_n(F)$ is invariant with respect to transposition. For
non-Archimedean $F$ the latter is a classical theorem of Gelfand
and Kazhdan.
\end{abstract}

\setcounter{tocdepth}{2}
 \tableofcontents

\section{Introduction}
Harish-Chandra developed a technique based on Jordan decomposition
that allows to reduce certain statements on conjugation invariant
distributions on a reductive group to the set of unipotent
elements, provided that the statement is known for certain
subgroups (see e.g. \cite{HCh}).

In this paper we generalize an aspect of this technique to the
setting of a reductive group acting on a smooth affine algebraic
variety, using the Luna Slice Theorem. Our technique is oriented
towards proving Gelfand property for pairs of reductive groups.

Our approach is uniform for all local fields of characteristic
zero -- both Archimedean and non-Archimedean.

\subsection{Main results}$ $\\
The core of this paper is Theorem \ref{Gen_HC}:
\begin{theorem*}
Let a reductive group $G$ act on a smooth affine variety $X$, both
defined over a local field $F$ of characteristic zero. Let $\chi$
be a character of $G(F)$.

Suppose that for any $x \in X(F)$ with closed orbit there are no
non-zero distributions on the normal space at $x$ to the orbit
$G(F)x$ which are $(G(F)_x,\chi)$-equivariant, where $G_x$ denotes
the stabilizer of $x$.

Then there are no non-zero $(G(F),\chi)$-equivariant distributions
on $X(F)$.
\end{theorem*}

In fact, a stronger version based on this theorem is given in
Corollary \ref{Strong_HC_Cor}. This stronger version is based on
an inductive argument. It shows that it is enough to prove that
there are no non-zero equivariant distributions on the normal
space to the orbit $G(F)x$ at $x$ under the assumption that all
such distributions are supported in a certain closed subset which
is the analog of the nilpotent cone.

We apply this stronger version to problems of the following type.
Let a reductive group $G$ act on a smooth affine variety $X$, and
$\tau$ be an involution of $X$ which normalizes the image of $G$
in $\Aut(X)$. We want to check whether any $G(F)$-invariant
distribution on $X(F)$ is also $\tau$-invariant. Evidently, there
is the following
necessary condition on $\tau$:\\
(*) Any closed orbit in $X(F)$ is $\tau$-invariant.\\
In some cases this condition is also sufficient. In these cases we
call the action of $G$ on $X$ \emph{tame}.

This is a weakening of the property called \emph{"density"} in
\cite{RR}. However, it is sufficient for the purpose of proving
Gelfand property for pairs of reductive groups.

In \S \ref{SecTame} we give criteria for tameness of actions. In
particular, we introduce the notion of \emph{"special"} action in
order to show that certain actions are tame (see Theorem
\ref{Invol_HC} and Proposition \ref{SpecWeakReg}). Also, in many
cases one can verify that an action is special using purely
algebraic-geometric means.

In the second part of the paper we restrict our attention to the
case of symmetric pairs. We transfer the terminology on actions to
terminology on symmetric pairs. For example, we call a symmetric
pair $(G,H)$ \emph{tame} if the action of $H \times H$ on $G$ is
tame.

In addition we introduce the notion of a \emph{"regular"}
symmetric pair (see Definition \ref{DefReg}), which also helps to
prove Gelfand property. Namely, we prove Theorem
\ref{GoodHerRegGK}.

\begin{theorem*}
Let $G$ be a reductive group defined over a local field $F$ and
let $\theta$ be an involution of $G$. Let $H:=G^{\theta}$ and let
$\sigma$ be the anti-involution defined by
$\sigma(g):=\theta(g^{-1})$. Consider the symmetric pair $(G,H)$.

Suppose that all its \emph{"descendants"} (including itself, see
Definition \ref{descendant}) are regular. Suppose also that any
closed $H(F)$-double coset in $G(F)$ is $\sigma$-invariant.

 Then every bi-$H(F)$-invariant distribution on $G(F)$ is
$\sigma$-invariant. In particular, by Gelfand-Kazhdan criterion,
the pair $(G,H)$ is a Gelfand pair (see \S \ref{Gel}).
\end{theorem*}

Also, we formulate an algebraic-geometric criterion for regularity
of a pair (Proposition \ref{SpecCrit}).  We sum up the various
properties of symmetric pairs and their interrelations in a
diagram in Appendix \ref{Diag}.

As an application and illustration of our methods we prove in \S
\ref{RJR} that the pair $(\GL_{n+k},\GL_n \times \GL_k)$ is a
Gelfand pair by proving that it is regular, along with its
descendants. In the non-Archimedean case this was proven in
\cite{JR} and our proof is along the same lines. Our technique
enabled us to streamline some of the computations in the proof of
\cite{JR} and to extend it to the Archimedean case.

 We also prove (in \S \ref{Sec2RegPairs}) that the
pair $(G(E),G(F))$ is tame for any reductive group $G$ over $F$
and a quadratic field extension $E/F$. This implies that the pair
$(\GL_n(E),\GL_n(F))$ is a Gelfand pair. In the non-Archimedean
case this was proven in \cite{Fli}. Also we prove that the adjoint
action of a reductive group on itself is tame. This is a
generalization of a classical theorem by Gelfand and Kazhdan, see
\cite{GK}.

In general, we conjecture that any symmetric pair is regular. This
would imply the van Dijk conjecture:

\begin{conjecture*}[van Dijk]
Any symmetric pair $(G,H)$ over $\C$ such that $G/H$ is connected
is a Gelfand pair.
\end{conjecture*}

\subsection{Related work}
$ $\\This paper was inspired by the paper \cite{JR} by Jacquet and
Rallis where they prove that the pair $(\GL_{n+k}(F),\GL_n(F)
\times \GL_k(F))$ is a Gelfand pair for any non-Archimedean local
field $F$ of characteristic zero. Our aim was to see to what
extent their techniques generalize.

Another generalization of Harish-Chandra descent using the Luna
Slice Theorem has been carried out in the non-Archimedean case in
\cite{RR}. In that paper Rader and Rallis investigated spherical
characters of $H$-distinguished representations of $G$ for
symmetric pairs $(G,H)$ and checked the validity of what they call
the \emph{"density principle"} for rank one symmetric pairs. They
found out that the principle usually holds, but also found
counterexamples.

In \cite{vD}, van-Dijk investigated rank one symmetric pairs in
the Archimedean case and classified the Gelfand pairs among them.
In \cite{Bos-vD}, van-Dijk and Bosman studied the non-Archimedean
case and obtained results for most rank one symmetric pairs. We
hope that the second part of our paper will enhance the
understanding of this question for symmetric pairs of higher rank.

\subsection{Structure of the paper} $ $\\
In \S \ref{Prel} we introduce notation and terminology which
allows us to speak uniformly about spaces of points of smooth
algebraic varieties over Archimedean and non-Archimedean local
fields, and equivariant distributions on those spaces.

In \S\S \ref{PrelLoc} we formulate a version of the Luna Slice
Theorem for points over local fields (Theorem \ref{LocLuna}). In
\S\S \ref{PrelDist} we formulate results on equivariant
distributions and equivariant Schwartz distributions. Most of
those results are borrowed from \cite{BZ}, \cite{Ber}, \cite{Bar}
and \cite{AGS1}, and the rest are proven in Appendix
\ref{AppSubFrob}.

In \S \ref{SecDescent} we formulate and prove the Generalized
Harish-Chandra Descent Theorem and its stronger version.

\S \ref{DistVerSch} is of interest only in the Archimedean case.
In that section we prove that in the cases at hand if there are no
equivariant Schwartz distributions then there are no equivariant
distributions at all. Schwartz distributions are discussed in
Appendix \ref{AppSubFrob}.

In \S \ref{SecFour} we formulate a homogeneity Theorem which helps
us to check the conditions of the Generalized Harish-Chandra
Descent Theorem. In the non-Archimedean case this theorem had been
proved earlier (see e.g. \cite{JR}, \cite{RS2} or \cite{AGRS}). We
provide the proof for the Archimedean case in Appendix
\ref{AppRealHom}.

In \S \ref{SecTame} we introduce the notion of tame actions and
provide tameness criteria.

In \S \ref{SecSymPairs} we apply our tools to symmetric pairs. In
\S\S \ref{SecTamePairs} we provide criteria for tameness of a
symmetric pair. In \S\S \ref{SecRegPairs} we introduce the notion
of a regular symmetric pair and prove Theorem \ref{GoodHerRegGK}
alluded to above. In \S\S \ref{conj} we discuss conjectures about
the regularity and the Gelfand property of symmetric pairs. In
\S\S \ref{Sec2RegPairs} we prove that certain symmetric pairs are
tame.
 In \S\S
\ref{RJR} we prove that the pair  $(\GL_{n+k}(F),\GL_n(F) \times
\GL_k(F))$ is regular.

In \S \ref{Gel} we recall basic facts on Gelfand pairs and their
connections to invariant distributions. We also prove that the
pairs $(\GL_{n+k}(F),\GL_n(F) \times \GL_k(F))$ and
$(\GL_n(E),\GL_n(F))$ are Gelfand pairs for any local field $F$
and its quadratic extension $E$.

We start Appendix \ref{AppLocField} by discussing different
versions of the Inverse Function Theorem for local fields. Then we
prove a version of the Luna Slice Theorem for points over local
fields (Theorem \ref{LocLuna}). For Archimedean $F$ this  was done
by Luna himself in \cite{Lun2}.

Appendices \ref{AppSubFrob} and \ref{AppRealHom} are of interest
only in the Archimedean case.

In Appendix \ref{AppSubFrob} we discuss Schwartz distributions on
Nash manifolds. We prove  Frobenius reciprocity for them and
construct the pullback of a Schwartz distribution under a Nash
submersion. Also we prove that $G$-invariant distributions which
are (Nashly) compactly supported modulo $G$ are Schwartz
distributions.

In Appendix \ref{AppRealHom} we prove the Archimedean version of
the Homogeneity Theorem discussed in \S \ref{SecFour}.

In Appendix \ref{SecRedLocPrin} we formulate and prove a version
of Bernstein's Localization Principle (Theorem \ref{LocPrin}).
This appendix is of interest only for Archimedean $F$ since for
$l$-spaces a more general version of this principle had been
proven in \cite{Ber}.  This appendix is used in \S
\ref{DistVerSch}.

In \cite{AGS2} we formulated Localization Principle in the setting
of differential geometry. Admittedly, we currently do not have a
proof of this principle in such a general setting. However, in
Appendix \ref{SecRedLocPrin} we present a proof in the case of a
reductive group $G$ acting on a smooth affine variety $X$. This
generality is sufficiently wide for all applications we
encountered up to now, including the one considered in
\cite{AGS2}.

Finally, in Appendix \ref{Diag} we present a diagram that
illustrates the interrelations of various properties of symmetric
pairs.

\subsection{Acknowledgements}
We would like to thank our teacher \textbf{Joseph Bernstein} for
our mathematical education.

We also thank \textbf{Vladimir Berkovich}, \textbf{Joseph
Bernstein}, \textbf{Gerrit van Dijk}, \textbf{Stephen Gelbart},
\textbf{Maria Gorelik}, \textbf{Herve Jacquet}, \textbf{David
Kazhdan}, \textbf{Erez Lapid}, \textbf{Shifra Reif}, \textbf{Eitan
Sayag}, \textbf{David Soudry}, \textbf{Yakov Varshavsky} and
\textbf{Oksana Yakimova} for fruitful discussions, and \textbf{Sun
Binyong}, \textbf{Gerard Schiffmann}, and the referees for useful remarks. \Rami{We thank \textbf{Shachar Carmeli}, \textbf{Yufeng Li}, \textbf{Junyan Xiao} and \textbf{Jun Yu} for finding some mistakes.}

Finally we thank \textbf{Anna Gourevitch} for the graphical design
of Appendix \ref{Diag}.

Both authors are partially supported by BSF grant, GIF grant, and
ISF Center of excellency grant.
\part{Generalized Harish-Chadra descent}
\section{Preliminaries and notation} \label{Prel}
\subsection{Conventions} \label{Conv}

\begin{itemize}
\item
Henceforth we fix a local field $F$ of characteristic zero. All
the algebraic varieties and algebraic groups that we will consider
will be defined over $F$.
\item For a group $G$ acting on a set $X$ we denote by $X^G$ the set of fixed points of $X$. Also, for an element $x \in X$
we denote by $G_x$ the stabilizer of $x$.
\item By a reductive group we mean a (non-necessarily connected) algebraic reductive group.
\item We consider an algebraic variety $X$ defined over $F$ as an algebraic
variety over $\oF$ together with action of the Galois group
$Gal(\oF /F)$. On $X$ we only consider  the Zariski topology. On
$X(F)$ we only consider  the analytic (Hausdorff) topology. We
treat finite-dimensional linear spaces defined over $F$ as
algebraic varieties.
\item The tangent space of a manifold (algebraic, analytic, etc.) $X$ at $x$
will be denoted by $T_xX$.
\item Usually we will use the letters $X, Y, Z, \Delta$ to denote algebraic
varieties and the letters $G, H$ to denote reductive groups. We
will usually use the letters $V, W,U,K,M,N,C,O,S,T$ to denote
analytic spaces (such as $F$-points of algebraic varieties) and
the letter $K$ to denote analytic groups. Also we will use the
letters $L, V, W$ to denote vector spaces of all kinds.
\end{itemize}
\subsection{Categorical quotient} \label{CatQuot}
\setcounter{lemma}{0}
\begin{definition}
Let an algebraic group $G$ act on an algebraic variety $X$. A pair
consisting of an algebraic variety $Y$ and a $G$-invariant
morphism $\pi:X \to Y$ is called \textbf{the quotient of $X$ by
the action of $G$} if for any pair $(\pi ', Y')$, there exists a
unique morphism $\phi : Y \to Y'$ such that $\pi ' = \phi \circ
\pi$. Clearly, if such pair exists it is unique up to a canonical
isomorphism. We will denote it by $(\pi _X, X/G)  $.
\end{definition}

\begin{theorem}[cf. \cite{Dre}] \label{Quotient}
Let a reductive group $G$ act on an affine variety $X$. Then the
quotient $X/G$ exists, and every fiber of the quotient map $\pi_X$
contains a unique closed orbit. In fact, $X/G:=\Spec\cO(X)^G.$
\end{theorem}

\subsection{Algebraic geometry over local fields}
\label{PrelLoc}

\subsubsection{Analytic manifolds}
$ $\\
In this paper we  consider distributions over $l$-spaces, smooth
manifolds and Nash manifolds. $l$-spaces are locally compact
totally disconnected topological spaces and Nash manifolds are
semi-algebraic smooth manifolds.

For basic facts on $l$-spaces and distributions over them we refer
the reader to \cite[\S 1]{BZ}.

For basic facts on Nash manifolds and Schwartz functions and
distributions over them see Appendix \ref{AppSubFrob} and
\cite{AG1}. In this paper we  consider only separated Nash
manifolds.

We  now introduce notation and terminology which  allows a uniform
treatment of the Archimedean and the non-Archimedean cases.

We will use the notion of an analytic manifold over a local field
(see e.g. \cite[Part II, Chapter III]{Ser}). When we say
"\textbf{analytic manifold}" we always mean analytic manifold over
some local field. Note that an analytic manifold over a
non-Archimedean field is in particular an $l$-space and an
analytic manifold over an Archimedean field is in particular a
smooth manifold.

\begin{definition}
A \textbf{B-analytic manifold} is either an analytic manifold over
a non-Archimedean local field, or a Nash manifold.
\end{definition}

\begin{remark} If $X$ is a smooth algebraic variety,
then $X(F)$ is a B-analytic manifold and $(T_xX)(F) = T_x(X(F)).$
\end{remark}

\begin{notation}
Let $M$ be an analytic manifold and $S$ be an analytic
submanifold. We denote by $N_S^M:=(T_M|_\Rami{S})/T_S $ the
\textbf{normal bundle to $S$ in $M$}. The \textbf{conormal bundle}
is defined by $CN_S^M:=(N_S^M)^*$. Denote by $\Sym^k(CN_S^M)$ the
k-th symmetric power of the conormal bundle. For a point $y\in S$
we denote by $N_{S,y}^M$ the normal space to $S$ in $M$ at the
point $y$ and by $CN_{S,y}^M$ the conormal space.
\end{notation}

\subsubsection{$G$-orbits on $X$ and $G(F)$-orbits on $X(F)$} 

\begin{lemma}[see  Appendix \ref{AppSub}] \label{OrbitIsOpen}
Let $G$ be an algebraic group and let $H \subset G$ be a closed
subgroup. Then $G(F)/H(F)$ is open and closed in $(G/H)(F)$.
\end{lemma}

\begin{corollary}
Let an algebraic group $G$ act on an algebraic variety $X$. Let $x
\in X(F)$. Then $$N_{Gx,x}^{X}(F) \cong N_{G(F)x,x}^{X(F)}.$$
\end{corollary}

\begin{proposition} \label{LocClosedOrbit}
Let an algebraic group $G$ act on an algebraic variety $X$.
Suppose that $S \subset X(F)$ is a non-empty closed
$G(F)$-invariant subset. Then $S$ contains a closed orbit.
\end{proposition}
\begin{proof}
The proof is by Noetherian induction on $X$. Choose $x \in S$.
Consider $Z:=\overline{Gx} -Gx$.

If $Z(F)\cap S$ is empty then $Gx(F) \cap S$ is closed and hence
$G(F)x \cap S$ is closed by Lemma \ref{OrbitIsOpen}. Therefore
$G(F)x$ is closed.

If $Z(F)\cap S$ is non-empty then $Z(F) \cap S$ contains a closed
orbit by the induction assumption.
\end{proof}

\begin{corollary} \label{OpenClosedAll}
Let an algebraic group $G$ act on an algebraic variety $X$. Let
$U$ be an open $G(F)$-invariant subset of $X(F)$. Suppose that $U$
contains all closed $G(F)$-orbits. Then $U=X(F)$.
\end{corollary}

\begin{theorem}[\cite{RR}, \S 2 fact A, pages 108-109] \label{LocZarClosed}%
Let a reductive group $G$ act on an affine variety $X$. Let $x
\in X(F)$. Then the following are equivalent:\\
(i) $G(F)x \subset X(F)$ is closed (in the analytic topology).\\
(ii) $Gx\subset X$ is closed (in the Zariski topology).
\end{theorem}

\begin{definition}
Let a reductive group $G$ act on an affine variety $X$. We call an
element $x \in X$ \textbf{$G$-semisimple} if its orbit $Gx$ is
closed.
\end{definition}

In particular, in the case where $G$ acts on itself by
conjugation, the notion of $G$-semisimplicity coincides with the
usual one.

\begin{notation}
Let $V$ be an $F$-rational finite-dimensional representation of a
reductive group $G$. We set
$$Q_G(V):=Q(V):=(V/V^G)(F).$$ Since $G$ is reductive, there is a
canonical embedding $Q(V) \hookrightarrow V(F)$. Let $\pi : V(F)
\to (V/G)(F)$ be the natural map. We set$$\Gamma_G(V):=
\Gamma(V):= \pi^{-1}(\pi(0)).$$ Note that $\Gamma(V) \subset
Q(V)$. We also set$$R_G(V):= R(V):= Q(V) - \Gamma(V).$$
\end{notation}

\begin{notation}
Let a reductive group $G$ act on an affine variety $X$. For a
$G$-semisimple element $x \in X(F)$ we set $$S_x := \{y \in X(F)
\, | \, \overline{G(F)y} \ni x\}.$$
\end{notation}

\begin{lemma} \label{Gamma}
Let $V$ be an  $F$-rational finite-dimensional representation of a
reductive group $G$. Then $\Gamma(V) = S_0$.
\end{lemma}
This lemma follows from \cite[fact A on page 108]{RR} for
non-Archimedean $F$ and \cite[Theorem 5.2 on page 459]{Brk} for
Archimedean $F$.

\begin{example}
Let a reductive group $G$ act on its Lie algebra $\g$ by the
adjoint action. Then $\Gamma(\g)$ is the set of nilpotent elements
of $\g$.
\end{example}

\begin{proposition}
Let a reductive group $G$ act on an affine variety $X$. Let $x,z
\in X(F)$ be $G$-semisimple elements which do not lie in the same
orbit of $G(F)$. Then there exist disjoint $G(F)$-invariant open
neighborhoods $U_x$ of $x$ and $U_z$ of $z$.
\end{proposition}
For the proof of this Proposition see \cite{Lun2} for Archimedean
$F$ and \cite[fact B on page 109]{RR} for non-Archimedean $F$.

\begin{corollary} \label{EquivClassClosed}
Let a reductive group $G$ act on an affine variety $X$. Suppose
that $x \in X(F)$ is a $G$-semisimple element. Then the set $S_x$
is closed.
\end{corollary}
\begin{proof}
Let $y \in \overline{S_x}$. By Proposition \ref{LocClosedOrbit},
$\overline{G(F)y}$ contains a closed orbit $G(F)z$. If $G(F)z =
G(F)x$ then $y \in S_x$. Otherwise, choose disjoint open
$G$-invariant neighborhoods $U_z$ of $z$ and $U_x$ of $x$. Since
$z \in \overline{G(F)y}$, $U_z$ intersects $G(F)y$ and hence
contains $y$. Since  $y \in \overline{S_x}$, this means that $U_z$
intersects $S_x$. Let $t \in U_z \cap S_x$. Since $U_z$ is
$G(F)$-invariant, $G(F)t \subset U_z$. By the definition of $S_x$,
$x \in \overline{G(F)t}$ and hence $x \in \overline{U_z}$. Hence
$U_z$ intersects $U_x$ -- contradiction!
\end{proof}

\subsubsection{Analytic Luna slices}

\begin{definition}
Let a reductive group $G$ act on an affine variety $X$. Let $\pi:
X(F) \to (X/G)(F)$ be the natural map. An open subset $U \subset
X(F)$ is called \textbf{saturated} if there exists an open subset
$V \subset (X/G)(F)$ such that $U = \pi^{-1}(V)$.
\end{definition}

\noindent We will use the following corollary of the Luna Slice
Theorem:

\begin{theorem}[see Appendix \ref{AppLun}] \label{LocLuna}
Let a reductive group $G$ act on a smooth affine variety $X$. Let
$x \in X(F)$ be $G$-semisimple. Consider the natural action of the
stabilizer $G_x$ on the normal space $N_{Gx,x}^{X}$.
Then there exist\\
(i) an open $G(F)$-invariant $B$-analytic neighborhood $U$ of
$G(F)x$ in $X(F)$ with a
$G$-equivariant $B$-analytic retract $p:U \to G(F)x$ and\\
(ii) a $G_x$-equivariant $B$-analytic embedding $\psi:p^{-1}(x)
\hookrightarrow N_{Gx,x}^{X}(F)$ with an open saturated image such
that $\psi(x)=0$.
\end{theorem}

\begin{definition}
In the notation of the previous theorem, denote $S:= p^{-1}(x)$
and $N:=N_{Gx,x}^{X}(F)$. We call the quintuple $(U,p,\psi,S,N)$
an \textbf{analytic Luna slice at $x$}.
\end{definition}

\begin{corollary} \label{LocLunCor}
In the notation of the previous theorem, let $y\in p^{-1}(x)$.
Denote $z:=\psi(y)$. Then\\
(i) $(G(F)_x)_z=G(F)_y$\\
(ii) $N_{G(F)y,y}^{X(F)} \cong N_{G(F)_x z, z}^{N}$ as
$G(F)_y$-spaces\\
(iii) $y$ is $G$-semisimple if and only if $z$ is
$G_x$-semisimple.
\end{corollary}

\subsection{Vector systems} \label{vs}
\footnote{Subsection \ref{vs} and in particular the notion of
"vector system" along with the results at the end of \S\S
\ref{GHC} and \S\S \ref{sGHC} are not essential for the rest of
the paper. They are merely included for future reference.}

In this subsection we introduce the term \emph{"vector system"}.
This term allows to formulate statements in wider generality.

\begin{definition}
For an analytic manifold $M$ we define the notions of a
\textbf{vector system} and a \textbf{B-vector system} over it.

For a smooth manifold $M$,  a vector system over $M$ is a pair
$(E,B)$ where $B$ is a smooth locally trivial fibration over $M$
and $E$ is a smooth (finite-dimensional) vector bundle over $B$.

For a Nash manifold $M$, a B-vector system over $M$ is a pair
$(E,B)$ where $B$ is a Nash fibration over $M$ and $E$ is a Nash
(finite-dimensional) vector bundle over $B$.

For an $l$-space $M$, a vector system over $M$ (or a B-vector
system over $M$) is a sheaf of complex linear spaces.
\end{definition}

In particular, in the case where $M$ is a point, a vector system
over $M$ is either a $\C$-vector space if $F$ is non-Archimedean,
or a smooth manifold together with a vector bundle in the case
where $F$ is Archimedean. The simplest example of a vector system
over a manifold $M$ is given by the following.

\begin{definition}
Let $\mathcal{V}$ be a vector system over a point $pt$. Let $M$ be
an analytic manifold. A \textbf{constant vector system with fiber
$\mathcal{V}$} is the pullback of $\mathcal{V}$ with respect to
the map $M \to pt$. We denote it by $\mathcal{V}_M$.
\end{definition}

\subsection{Distributions} \label{PrelDist}

\begin{definition}
Let $M$ be an analytic manifold over $F$. We define
$C_c^{\infty}(M)$ in the following way.

If $F$ is non-Archimedean then $C_c^{\infty}(M)$ is the space of
locally constant compactly supported complex valued functions on
$M$. We do not consider any topology on $C_c^{\infty}(M)$.

If $F$ is Archimedean then $C_c^{\infty}(M)$ is the space of
smooth compactly supported complex valued functions on $M$,
endowed with the standard topology.

For any analytic manifold $M$, we define the space of
distributions $\cD(M)$ by $\cD(M):=C_c^{\infty}(M)^*$. We consider
the weak topology on it.
\end{definition}

\begin{definition}
Let $M$ be a $B$-analytic manifold. We define $\Sc(M)$ in the
following way.

If $M$ is an analytic manifold over non-Archimedean field,
$\Sc(M):=C_c^{\infty}(M)$.

If $M$ is a Nash manifold, $\Sc(M)$ is the space of Schwartz
functions on $M$, namely smooth functions which are rapidly
decreasing together with all their derivatives. See \cite{AG1} for
the precise definition. We consider $\Sc(M)$ as a \Fre space.

For any $B$-analytic manifold $M$, we define the space of
\textbf{Schwartz distributions} $\Sc^*(M)$ by
$\Sc^*(M):=\Sc(M)^*$. Clearly, $\Sc(M)^*$ is naturally embedded
into $\cD(M)$.
\end{definition}

\begin{notation}
Let $M$ be an analytic manifold. For a distribution $\xi \in
\cD(M)$ we denote by $\Supp(\xi)$ the support of $\xi$.

For a closed subset $N \subset M$ we denote
$$\cD_M(N):= \{\xi \in \cD(M)|\Supp(\xi) \subset N\}.$$
More generally, for a locally closed subset $N \subset M$ we
denote
$$\cD_M(N):=\cD_{M\setminus (\overline{N} \setminus N)}(N).$$

Similarly if $M$ is a $B$-analytic manifold and $N$ is a locally
closed subset we define $\Sc^*_M(N)$ in a similar vein.
\footnote{In the Archimedean case, locally closed is considered
with respect to the restricted topology -- cf.~Appendix
\ref{AppSubFrob}.}
\end{notation}

\begin{definition}
Let $M$ be an analytic manifold over $F$ and $\E$ be a vector
system over $M$. We define $C_c^{\infty}(M,\E)$ in the following
way.

If $F$ is non-Archimedean then $C_c^{\infty}(M,\E)$ is the space
of compactly supported sections of $\E$.

If $F$ is Archimedean and $\E = (E,B)$ where $B$ is a fibration
over $M$ and $E$ is a vector bundle over $B$, then
$C_c^{\infty}(M,\E)$ is the complexification of the space of
smooth compactly supported sections of $E$ over $B$.

If $\mathcal{V}$ is a vector system over a point then we denote
$C_c^{\infty}(M,\mathcal{V}) := C_c^{\infty}(M,\mathcal{V}_M)$.
\end{definition}

We define $\cD(M,\E)$, $\cD_M(N,\E)$, $\Sc(M,\E)$, $\Sc^*(M,\E)$
and $\Sc^*_M(N,\E)$ in the natural way.

\begin{theorem}\label{Filt_nonarch}
Let an $l$-group $K$ act on an $l$-space $M$. Let $M =
\bigcup_{i=0}^l M_i$ be a $K$-invariant stratification of $M$. Let
$\chi$ be a character of $K$. Suppose that
$\Sc^*(M_i)^{K,\chi}=0$. Then $\Sc^*(M)^{K,\chi}=0$.
\end{theorem}
This theorem is a direct corollary of \cite[Corollary 1.9]{BZ}.


For the proof of the next theorem see e.g. \cite[\S B.2]{AGS1}.
\begin{theorem} \label{NashFilt}
Let a Nash group $K$ act on a Nash manifold $M$. Let $N$ be a
locally closed subset. Let $N = \bigcup_{i=0}^l N_i$ be a Nash
$K$-invariant stratification of $N$. Let $\chi$ be a character of
$K$. Suppose that for any $k \in \Z_{\geq 0}$ and $0 \leq i \leq
l$, $$\Sc^*(N_i,\Sym^k(CN_{N_i}^M))^{K,\chi}=0.$$ Then
$\Sc^*_M(N)^{K,\chi}=0.$
\end{theorem}

\begin{theorem}[Frobenius reciprocity] \label{Frob}
Let an analytic group $K$ act on an analytic manifold $M$. Let $N$
be an analytic manifold with a transitive action of $K$. Let
$\phi:M \to N$ be a $K$-equivariant map.

Let $z \in N$ be a point and $M_z:= \phi^{-1}(z)$ be its fiber.
Let $K_z$ be the stabilizer of $z$ in $K$. Let $\Delta_K$ and
$\Delta_{K_z}$ be the modular characters of $K$ and $K_z$.

Let $\E$ be a $K$-equivariant vector system over $M$. Then\\
(i) there exists a canonical isomorphism $$\Fr: \cD(M_z,\E|_{M_z}
\otimes \Delta_K|_{K_z} \cdot \Delta_{K_z}^{-1})^{K_z} \cong
\cD(M,\E)^K.$$ In particular, $\Fr$ commutes with restrictions to
open sets.

(ii) For B-analytic manifolds $\Fr$ maps $\Sc^*(M_z,\E|_{M_z}
\otimes \Delta_K|_{K_z} \cdot \Delta_{K_z}^{-1})^{K_z}$ to
$\Sc^*(M,\E)^K$.
\end{theorem}

For the proof of (i) see \cite[\S\S 1.5]{Ber}  and \cite[\S\S 2.21
- 2.36]{BZ}  for the case of $l$-spaces and \cite[Theorem
4.2.3]{AGS1} or \cite{Bar} for smooth manifolds. For the proof of
(ii) see Appendix \ref{AppSubFrob}.

We will also use the following straightforward proposition.

\begin{proposition} \label{Product}
Let $K_i$ be analytic groups acting on analytic manifolds $M_i$
for $i=1 \ldots n$. Let $\Omega_i \subset K_i$ be analytic
subgroups. Let $\E_i \to M_i$ be $K_i$-equivariant vector systems.
Suppose that
$$\cD(M_i,E_i)^{\Omega_i}=\cD(M_i,E_i)^{K_i}$$ for all $i$. Then
$$\cD(\prod M_i, \boxtimes E_i)^{\prod \Omega_i}=\cD(\prod M_i,
\boxtimes E_i)^{\prod K_i},$$ where $\boxtimes$ denotes the
external product.

Moreover, if $\, \Omega_i$, $K_i$, $M_i$ and $\E_i$ are
$B$-analytic then the analogous statement holds for Schwartz
distributions.
\end{proposition}

For the proof see e.g. \cite[proof of Proposition 3.1.5]{AGS1}.

\section{Generalized Harish-Chandra descent} \label{SecDescent}

\subsection{Generalized Harish-Chandra descent}
\label{GHC} $ $\\ In this subsection we will prove the following
theorem.

\begin{theorem} \label{Gen_HC}
Let a reductive group $G$ act on a smooth affine variety $X$. Let
$\chi$ be a character of $G(F)$. Suppose that for any
$G$-semisimple $x \in X(F)$ we have
$$\cD(N_{Gx,x}^X(F))^{G(F)_x,\chi}=0.$$ Then
$$\cD(X(F))^{G(F),\chi}=0.$$
\end{theorem}

\begin{remark}
In fact, the converse is also true. We will not prove it since we
will not use it.
\end{remark}

For the proof of this theorem we will need the following lemma

\begin{lemma}
Let a reductive group $G$ act on a smooth affine variety $X$. Let
$\chi$ be a character of $G(F)$. Let $U \subset X(F)$ be an open
saturated subset. Suppose that $\cD(X(F))^{G(F),\chi}=0.$ Then
$\cD(U)^{G(F),\chi}=0.$
\end{lemma}
\begin{proof}

Consider the quotient $X/G$. It is an affine algebraic variety.
Embed it in an affine space $\mathbb{A}^n$. This defines a map
$\pi:X(F) \to F^n$. Since $U$ is saturated, there exists an open
subset $V \subset (X/G)(F)$ such that $U=\pi^{-1}(V)$. Clearly
there exists an open subset $V' \subset F^n$ such that $V' \cap
(X/G)(F) = V$.

Let $\xi \in \cD(U)^{G(F),\chi}$. Suppose that $\xi$ is non-zero.
Let $x \in \Supp \xi$ and let $y:=\pi(x)$. Let $g \in
C^{\infty}_c(V')$ be such that $g(y) = 1$. Consider $\xi' \in
\cD(X(F))$ defined by $\xi'(f) := \xi(f \cdot (g \circ \pi))$.
Clearly, $\Supp(\xi') \subset U$ and hence we can interpret $\xi'$
as an element in $\cD(X(F))^{G(F),\chi}$. Therefore $\xi'=0$. On
the other hand, $x \in \Supp(\xi')$. Contradiction.
\end{proof}

\begin{proof}[Proof of Theorem \ref{Gen_HC}.] Let $x$
be a $G$-semisimple element. Let ($U_x$,$p_x$,$\psi_x$,
$S_x$,$N_x$) be an analytic Luna slice at $x$. \Rami{Let $\xi\in \cD(X(F))^{G(F),\chi}$.}

Let $\xi' = \xi |_{U_x}$. Then $\xi' \in \cD(U_x)^{G(F),\chi}.$ By
Frobenius reciprocity it corresponds to $\xi'' \in
\cD(S_x)^{G_x(F),\chi}$.

The distribution $\xi''$ corresponds to a distribution $\xi''' \in
\cD(\psi_x(S_x))^{G_x(F),\chi}.$

However, by the previous lemma the assumption implies that
$\cD(\psi_x(S_x))^{G_x(F),\chi}=0.$ Hence $\xi'=0$.

Let $S \subset X(F)$ be the set of all $G$-semisimple points. Let
$U= \bigcup_{x\in S} U_x$. We saw that $\xi|_U=0$. On the other
hand, $U$ includes all the closed orbits, and hence by Corollary
\ref{OpenClosedAll} $U=X\Rami{(F)}$.
\end{proof}

The following generalization of this theorem is proven in the same
way.

\begin{theorem} \label{Gen_HC_K}
Let a reductive group $G$ act on a smooth affine variety $X$. Let
$K \subset G(F)$ be an open subgroup and let $\chi$ be a character
of $K$. Suppose that for any $G$-semisimple $x \in X(F)$ we have
$$\cD(N_{Gx,x}^X(F))^{K_x,\chi}=0.$$ Then
$$\cD(X(F))^{K,\chi}=0.$$
\end{theorem}

Now we would like to formulate a slightly more general version of
this theorem concerning $K$-equivariant vector systems.
\footnote{Subsection \ref{vs} and in particular the notion of
"vector system" along with the results at the end of \S\S
\ref{GHC} and \S\S \ref{sGHC} are not essential for the rest of
the paper. They are merely included for future reference.}

\begin{definition}
Let a reductive group $G$ act on a smooth affine variety $X$. Let
$K \subset G(F)$ be an open subgroup. Let $\E$ be a
$K$-equivariant vector system on $X(F)$. Let $x \in X(F)$ be
$G$-semisimple. Let $\E'$ be a $K_x$-equivariant vector system on
$N_{Gx,x}^X(F)$. We say that $\E$ and $\E'$ are
\textbf{compatible} if there exists an analytic Luna slice
$(U,p,\psi,S,N)$ such that $\E|_S = \psi^*(\E')$.
\end{definition}
Note that if $\E$ and $\E'$ are constant with the same fiber then
they are compatible.

The following theorem is proven in the same way as Theorem
\ref{Gen_HC}.
\begin{theorem} \label{Gen_HC_Sys}
Let a reductive group $G$ act on a smooth affine variety $X$. Let
$K \subset G(F)$ be an open subgroup and let $\E$ be a
$K$-equivariant vector system on $X(F)$. Suppose that for any
$G$-semisimple $x \in X(F)$ there exists a $K$-equivariant vector
system $\E'$ on $N_{Gx,x}^X(F)$, compatible with $\E$ such that
$$\cD(N_{Gx,x}^X(F),\E')^{K_x}=0.$$ Then
$$\cD(X(F),\E)^{K}=0.$$
\end{theorem}

If $\E$ and $\E'$ are B-vector systems and $K$ is an open
B-analytic subgroup\footnote{In fact, any open subgroup of a
B-analytic group is B-analytic.} of $G(F)$ then the theorem also
holds for Schwartz distributions. Namely, if
$\Sc^*(N_{Gx,x}^X(F),\E')^{K_x}=0$ for any $G$-semisimple $x \in
X(F)$ then $\Sc^*(X(F),\E)^{K}=0$. The proof is the same.

\subsection{A stronger version} \label{sGHC}
$ $\\
In this section we provide means to validate the conditions of
Theorems \ref{Gen_HC}, \ref{Gen_HC_K} and \ref{Gen_HC_Sys} based
on an inductive argument.

More precisely, the goal of this section is to prove the following
theorem.

\begin{theorem} \label{Strong_HC}
Let a reductive group $G$ act on a smooth affine variety $X$. Let
$K \subset G(F)$ be an open subgroup and let $\chi$ be a character
of $K$. Suppose that for any $G$-semisimple $x \in X(F)$ such that
$$\cD(R_{G_x}(N_{Gx,x}^X))^{K_x,\chi}=0$$ we have
$$\cD(Q_{G_x}(N_{Gx,x}^X))^{K_x,\chi}=0.$$ Then for any $G$-semisimple $x \in X(F)$ we have $$\cD(N_{Gx,x}^X(F))^{K_x,\chi}=0.$$
\end{theorem}

 Together with Theorem \ref{Gen_HC_K}, this theorem gives the
following corollary.
\begin{corollary} \label{Strong_HC_Cor}
Let a reductive group $G$ act on a smooth affine variety $X$. Let
$K \subset G(F)$ be an open subgroup and let $\chi$ be a character
of $K$. Suppose that for any $G$-semisimple $x \in X(F)$ such that
$$\cD(R(N_{Gx,x}^X))^{K_x,\chi}=0$$ we have
$$\cD(Q(N_{Gx,x}^X))^{K_x,\chi}=0.$$ Then $\cD(X(F))^{K,\chi}=0.$
\end{corollary}

From now till the end of the section we fix $G$, $X$, $K$ and
$\chi$. Let us introduce several definitions and notation.

\begin{notation}
Denote \itemize  \item  $T \subset X(F)$ the set of all
$G$-semisimple points.
\item For $x,y \in T$ we say that $x>y$ if $G_x\supsetneqq G_y$.
\item $T_0 := \{x \in T\ |\ \cD(Q(N_{Gx,x}^X))^{K_x,\chi}=0 \} = \{x \in T\ |\ \cD((N_{Gx,x}^X))^{K_x,\chi}=0 \} .$
\end{notation}

\begin{proof}[Proof of Theorem \ref{Strong_HC}]
We have to show that $T=T_0$. Assume the contrary.

Note that every chain in $T$ with respect to our ordering has a
minimum. Hence by Zorn's lemma every non-empty set in $T$ has a
minimal element. Let $x$ be a minimal element of $T - T_0$. To get
a contradiction, it is enough to show that
$\cD(R(N_{Gx,x}^X))^{K_x,\chi}=0$.

Denote $R:=R(N_{Gx,x}^X)$. By Theorem \ref{Gen_HC_K}, it is enough
to show that for any $y\in R$ we have
$$\cD(N_{G(F)_x y,y}^R)^{(K_x)_y,\chi}=0.$$

Let $(U,p,\psi,S,N)$ be an analytic Luna slice at $x$.

Since $\psi(S)$ is open and contains 0, we can assume, upon
replacing $y$ by $\lambda y$ for some $\lambda \in F^{\times}$,
that $y \in \psi(S)$. Let $z \in S$ be such that $\psi(z)=y$. By
Corollary \ref{LocLunCor}, $G(F)_z=(G(F)_x)_y\subsetneqq G(F)_x$
and $N_{G(F)_x y,y}^R \Rami{\oplus (N_{Gx,x}^X)^{G_x}} \cong N_{G z,z}^X(F)$. Hence $(K_x)_y = K_z$
and therefore \Rami{it is enough to show that:
$$\cD(N_{G
z,z}^X(F))^{K_z,\chi}=0.$$}

However $z < x$ and hence $z \in T_0$ which means that $\cD(N_{G
z,z}^X(F))^{K_z,\chi}=0$.
\end{proof}

\begin{remark}
One can rewrite this proof such that it will use Zorn's lemma for
finite sets only, which does not depend on the axiom of choice.
\end{remark}

\begin{remark}
As before, Theorem \ref{Strong_HC} and Corollary
\ref{Strong_HC_Cor} also hold  for Schwartz distributions, with a
similar proof.
\end{remark}

Again, we can formulate a more general version of Corollary
\ref{Strong_HC_Cor} concerning vector systems.
\footnote{Subsection \ref{vs} and in particular, the notion of
"vector system" along with the results at the end of \S\S
\ref{GHC} and \S\S \ref{sGHC} are not essential for the rest of
the paper. They are merely included for future reference.}

\begin{theorem} \label{Strong_HC_Sys}
Let a reductive group $G$ act on a smooth affine variety $X$. Let
$K \subset G(F)$ be an open subgroup and let $\E$ be a
$K$-equivariant vector system on $X(F)$.

Suppose that for any $G$-semisimple $x \in X(F)$ satisfying\\
(*) for any $K_x \times F^{\times}$-equivariant vector system
$\E'$ on $R(N_{Gx,x}^X)$ (where $F^{\times}$ acts by homothety)
compatible with $\E$ we have $\cD(R(N_{Gx,x}^X),\E')^{K_x}=0$,

the following holds\\
(**)  there exists a $K_x \times F^{\times}$-equivariant vector
system $\E'$ on $Q(N_{Gx,x}^X)$ compatible with $\E$ such that
$$\cD(Q(N_{Gx,x}^X),\E')^{K_x}=0.$$

Then $\cD(X(F),\E)^{K}=0.$
\end{theorem}

The proof is the same as the proof of Theorem \ref{Strong_HC}
using the following lemma which follows from the definitions.

\begin{lemma}
Let a reductive group $G$ act on a smooth affine variety $X$. Let
$K \subset G(F)$ be an open subgroup and let $\E$ be a
$K$-equivariant vector system on $X(F)$. Let $x \in X(F)$ be
$G$-semisimple. Let $(U, p, \psi, S, N)$ be an analytic Luna slice
at $x$.

Let $\E'$ be a $K_x$-equivariant vector system on $N$ compatible
with $\E$. Let $y \in S$ be $G$-semisimple, and let $z :=
\psi(y)$. Let $\E''$ be a $(K_x)_z$-equivariant vector system on
$N_{G_xz,z}^N$ compatible with $\E'$. Consider the isomorphism
$N_{G_xz,z}^N(F) \cong N_{Gy,y}^X(F)$ and let $\E'''$ be the
corresponding $K_y$-equivariant vector system on $N_{Gy,y}^X(F)$.

Then $\E'''$ is compatible with $\E$.
\end{lemma}

Again, if $\E$ and $\E'$ are B-vector systems then the theorem
holds also for Schwartz distributions.

\section{Distributions versus Schwartz distributions}
\label{DistVerSch}

\setcounter{lemma}{0}

In this section $F$ is Archimedean. The tools developed in the
previous section enable us to prove the following version of the
Localization Principle.

\begin{theorem}[Localization Principle] \label{LocPrin}
Let a reductive group $G$ act on a smooth algebraic variety $X$.
Let $Y$ be an algebraic variety and $\phi:X \to Y$ be an affine
algebraic $G$-invariant map. Let $\chi$ be a character of $G(F)$.
Suppose that for any $y \in Y(F)$ we have
$\cD_{X(F)}((\phi^{-1}(y))(F))^{G(F),\chi}=0$. Then
$\cD(X(F))^{G(F),\chi}=0$.
\end{theorem}

For the proof see Appendix \ref{SecRedLocPrin}.

In this section we use this theorem to show that if there are no
$G(F)$-equivariant Schwartz distributions on $X(F)$ then there are
no $G(F)$-equivariant distributions on $X(F)$.

\begin{theorem}   \label{NoSNoDist}
Let a reductive group $G$ act on a smooth affine variety $X$. Let
$V$ be a finite-dimensional algebraic representation of $G(F)$.
Suppose that $$\Sc^*(X(F),V)^{G(F)}=0.$$ Then
$$\cD(X(F),V)^{G(F)}=0.$$
\end{theorem}

For the proof we will need the following definition and theorem.

\begin{definition} $ $

(i) Let a topological group $K$ act on a topological space $M$. We
call a closed $K$-invariant subset $C \subset M$ \textbf{compact
modulo $K$} if there exists a compact subset $C' \subset M$ such
that $C \subset KC'.$

(ii) Let a Nash group $K$ act on a Nash manifold $M$. We call a
closed $K$-invariant subset $C \subset M$ \textbf{Nashly compact
modulo $K$} if there exist a compact subset $C' \subset M$ and
semi-algebraic closed subset $Z \subset M$ such that $C \subset Z
\subset KC'.$
\end{definition}

\begin{remark}
Let a reductive group $G$ act on a smooth affine variety $X$. Let
$K:=G(F)$ and $M:=X(F)$. Then it is easy to see that the notions
of compact modulo $K$ and Nashly compact modulo $K$ coincide.
\end{remark}

\begin{theorem} \label{CompSchwartz}
Let a Nash group $K$ act on a Nash manifold $M$. Let $E$ be a
$K$-equivariant Nash bundle over $M$. Let $\xi \in \cD(M,E)^K$ be
such that $\Supp(\xi)$ is Nashly compact modulo $K$. Then $\xi \in
\Sc^*(M,E)^K$.
\end{theorem}

The statement and the idea of the proof of this theorem are due to
J. Bernstein. For the proof see Appendix \ref{KInvAreSchwartz}.

\begin{proof}[Proof of Theorem  \ref{NoSNoDist}]
Fix any $y \in (X/G)(F)$ and denote $M:=\pi_X^{-1}(y)(F)$.

By the Localization Principle (Theorem \ref{LocPrin} and Remark
\ref{RemLocVectSys}), it is enough to prove that
$$\Sc^*_{X(F)}(M,V)^{G(F)}=\cD_{X(F)}(M,V)^{G(F)}.$$ Choose $\xi
\in \cD_{X(F)}(M,V)^{G(F)}$. $M$ has a unique closed stable
$G$-orbit and hence a finite number of closed $G(F)$-orbits. By
Theorem \ref{CompSchwartz}, it is enough to show that $M$ is
Nashly compact modulo $G(F)$. Clearly $M$ is semi-algebraic.
Choose representatives $x_i$ of the closed $G(F)$-orbits in $M$.
Choose compact neighborhoods $C_i$ of $x_i$. Let $C' := \bigcup
C_i$. By Corollary \ref{OpenClosedAll}, $G(F)C' \supset M$.
\end{proof}

\section{Applications of Fourier transform and the Weil
representation} \label{SecFour}

Let $G$ be a reductive group   and $V$ be a finite-dimensional
 $F$-rational representation of $G$. Let $\chi$ be a
character of $G(F)$. In this section we provide some tools to
verify that $\Sc^*(Q(V))^{G(F),\chi}=0$ provided that
$\Sc^*(R(V))^{G(F),\chi}=0$.

\subsection{Preliminaries}
$ $\\
For this subsection let $B$ be a non-degenerate bilinear form on a
finite-dimensional vector space $V$ over $F$. We also fix an
additive character $\kappa$ of $F$. If $F$ is Archimedean we take
$\kappa(x):=e^{2\pi \mathrm{i} \re(x)}$.

\begin{notation}
We identify $V$ and $V^*$ via $B$ and endow $V$ with the self-dual
Haar measure with respect to $\Rami{\kappa}$. Denote by $\Fou_B: \Sc^*(V)
\to \Sc^*(V)$ the Fourier transform. For any B-analytic manifold
$M$ over $F$ we also denote by $\Fou_B:\Sc^*(M \times V) \to
\Sc^*(M \times V)$ the partial Fourier transform.
\end{notation}

\begin{notation}
Consider the homothety action of $F^{\times}$ on $V$ given by
$\rho(\lambda)v:= \lambda^{-1}v$. It gives rise to an action
$\rho$ of $F^{\times}$ on $\Sc^*(V)$.

Let $|\cdot |$ denote the normalized absolute value. Recall that
for $F =\R$, $|\lambda|$ is equal to the classical absolute value
but for $F =\C$, $|\lambda| = (\re \lambda)^2+ (\Im \lambda)^2$.
\end{notation}

\begin{notation}
We denote by $\gamma(B)$ the Weil constant.  For its definition
see e.g. \cite[\S 2.3]{Gel}  for non-Archimedean $F$ and \cite[\S
1]{RS1} for Archimedean $F$.

For any $t\in F^{\times}$ denote $\delta_B(t)=
\gamma(B)/\gamma(tB)$.
\end{notation}

Note that $\gamma(B)$ is an 8-th root of unity and if $\dim V$ is
odd and $F \neq \C$ then $\delta_B$ is \textbf{not} a
multiplicative character.

\begin{notation}
We denote $$Z(B):=\{x \in V\ |\ B(x,x)=0 \}.$$
\end{notation}

\begin{theorem} [non-Archimedean homogeneity] \label{NonArchHom}
Suppose that $F$ is \textbf{non-Archimedean}. Let $M$ be a
B-analytic manifold over $F$. Let $\xi \in \Sc^*_{V\times
M}(Z(B)\times M)$ be such that $\Fou_B(\xi) \in \Sc^*_{V\times
M}(Z(B)\times M)$. Then for any $t \in F^{\times}$, we have
$\rho(t)\xi=\delta_B(t) |t|^{\dim V/2} \xi$ and $\xi=
\gamma(B)^{-1} \Fou_B(\xi)$. In particular, if $\dim V$ is odd
then $\xi = 0$.
\end{theorem}
For the proof see e.g. \cite[\S\S 8.1]{RS2} or \cite[\S\S
3.1]{JR}.

For the Archimedean version of this theorem we will need the
following definition.
\begin{definition}
Let $M$ be a B-analytic manifold over $F$. We say that a
distribution $\xi \in \Sc^*(V \times M)$ is
\textbf{adapted to $B$} if either \\
(i) for any $t \in F^{\times}$ we have $\rho(t)\xi=
\delta(t)|t|^{\dim V/2} \xi$ and $ \xi$ is proportional to $\Fou_B \xi$ or\\
(ii) $F$ is Archimedean and for any $t \in F^{\times}$ we have
$\rho(t)\xi= \delta(t)t|t|^{\dim V/2} \xi$.\\
\Rami{%
(iii) $F=\C$ and either for any $t \in F^{\times}$ we have
$\rho(t)\xi= \delta(t)\bar t|t|^{\dim V/2} \xi$ or for any $t \in F^{\times}$ we have
$\rho(t)\xi= \delta(t)\frac{\bar t}{t}|t|^{\dim V/2} \xi$}
\end{definition}

Note that if $\dim V$ is odd and $F \neq \C$ then every
$B$-adapted distribution is zero.

\begin{theorem} [Archimedean homogeneity] \label{ArchHom}
Let $M$ be a Nash manifold. Let $L\subset \Sc^*_{V\times
M}(Z(B)\times M)$ be a non-zero subspace such that for all $\xi
\in L $ we have $\Fou_B(\xi) \in L$ and $B \cdot \xi \in L$ (here
$B$ is viewed as a quadratic function).

Then there exists a non-zero distribution $\xi \in L$ which is
adapted to $B$.
\end{theorem}
For Archimedean $F$ we prove this theorem in Appendix
\ref{AppRealHom}. For non-Archimedean $F$ it follows from Theorem
\ref{NonArchHom}.

We will also use the following trivial observation.

\begin{lemma}
Let a B-analytic group $K$ act linearly on $V$ and preserving $B$.
Let $M$ be a B-analytic $K$-manifold over $F$. Let $\xi \in
\Sc^*(V \times M)$ be a $K$-invariant distribution. Then
$\Fou_B(\xi)$ is also $K$-invariant.
\end{lemma}

\subsection{Applications}
$ $\\
The following two theorems easily follow \Rami{from} the results of the
previous subsection.

\begin{theorem} \label{Non_arch_Homog}
Suppose that $F$ is non-Archimedean. Let $G$ be a reductive group.
Let $V$ be a finite-dimensional  $F$-rational representation of
$G$. Let $\chi$ be character of $G(F)$. Suppose that
$\Sc^*(R(V))^{G(F),\chi}=0$. Let $V = V_1 \oplus V_2$ be a
$G$-invariant decomposition of $V$. Let $B$ be a $G$-invariant
symmetric non-degenerate bilinear form on $V_1$. Consider the
action $\rho$ of $F^{\times}$ on $V$ by homothety on $V_1$.

Then any $\xi \in \Sc^*(Q(V))^{G(F),\chi}$ satisfies
$\rho(t)\xi=\delta_B(t) |t|^{\dim \Rami{Q(V_1)}/2} \xi$ and $\xi= \gamma^{\Rami{-1}}(B)
\Fou_B \xi$. In particular, if $\dim V_1$ is odd then $\xi = 0$.
\end{theorem}

\begin{theorem} \label{Homog}
Let $G$ be a reductive group. Let $V$ be a finite-dimensional
 $F$-rational representation of $G$. Let $\chi$ be character
of $G(F)$. Suppose that $\Sc^*(R(V))^{G(F),\chi}=0$. Let $Q(V) = W
\oplus (\bigoplus_{i=1}^kV_i)$ be a $G$-invariant decomposition of
$Q(V)$. Let $B_i$ be $G$-invariant symmetric non-degenerate
bilinear forms on $V_i$. Suppose that any $\xi \in
\Sc^*_{Q(V)}(\Gamma(V))^{G(F),\chi}$ which is adapted to each
$B_i$ is zero.

Then $\Sc^*(Q(V))^{G(F),\chi}=0.$
\end{theorem}
\begin{remark}
One can easily generalize Theorems \ref{Homog} and
\ref{Non_arch_Homog} to the case of constant vector systems.
\end{remark}

\section{Tame actions} \label{SecTame}
\setcounter{lemma}{0}

In this section we consider problems of the following type. A
reductive group $G$ acts on a smooth affine variety $X$, and
$\tau$ is an automorphism of $X$ which normalizes the image of $G$
in $\Aut(X)$. We want to check whether any $G(F)$-invariant
Schwartz distribution on $X(F)$ is also $\tau$-invariant.

\begin{definition}
Let $\pi$ be an action of a reductive group $G$ on a smooth affine
variety $X$.
We say that an algebraic automorphism $\tau$ of $X$ is \textbf{$G$-admissible} if \\
(i) $\tau$ normalizes $\pi(G(F))$ and $\tau^2\in\pi(G(F))$.\\ (ii)
For any closed $G(F)$-orbit $O \subset X(F)$, we have $\tau(O)=O$.
\end{definition}

\begin{proposition} \label{AdmisDescends}
Let $\pi$ be an action of a reductive group $G$ on a smooth affine
variety $X$. Let $\tau$ be a $G$-admissible automorphism of $X$.
Let $K:=\pi(G(F))$ and let $\widetilde{K}$ be the group generated
by $\pi(G(F))$ and $\tau$. Let $x \in X(F)$ be a point with closed
$G(F)$-orbit. Let $\tau' \in \widetilde{K}_x - K_x$. Then $d\tau'
|_{N_{Gx,x}^X}$ is $G_x$-admissible.
\end{proposition}

\begin{proof}
Let $\widetilde{G}$ denote the group generated by $\pi(G)$ and
$\tau$. We check that the two properties of $G_x$-admissibility
hold for $d\tau'|_{N_{Gx,x}^X}$. The first one is obvious. For the
second, let $y \in N_{Gx,x}^X(F)$ be an element with closed
$G_x$-orbit. Let $y' = d\tau'(y)$. We have to show that there
exists $g \in G_x(F)$ such that $gy = y'$. Let $(U,p,\psi,S,N)$ be
an analytic Luna slice at $x$ with respect to the action of
$\widetilde{G}$. We can assume that there exists $z \in S$ such
that $y=\psi(z)$. Let $z' = \tau'(z)$. By Corollary
\ref{LocLunCor}, $z$ is $G$-semisimple. Since $\tau$ is
admissible, this implies that there exists $g \in G(F)$ such that
$gz = z'$. Clearly, $g \in G_x(F)$ and $gy = y'$.
\end{proof}

\begin{definition}
We call an action of a reductive group $G$ on a smooth affine
variety $X$ \textbf{tame} if for any $G$-admissible $\tau : X \to
X$, we have $\Sc^*(X(F))^{G(F)} \subset \Sc^*(X(F))^{\tau}.$
\end{definition}

\begin{definition}
We call an  $F$-rational representation $V$ of a reductive group
$G$ \textbf{linearly tame} if for any $G$-admissible linear map
$\tau : V \to V$, we have $\Sc^*(V(F))^{G(F)} \subset
\Sc^*(V(F))^{\tau}.$

We call a representation \textbf{weakly linearly tame} if for any
$G$-admissible linear map $\tau : V \to V$, such that
$\Sc^*(R(V))^{G(F)} \subset \Sc^*(R(V))^{\tau}$ we have
$\Sc^*(Q(V))^{G(F)} \subset \Sc^*(Q(V))^{\tau}.$
\end{definition}

\begin{theorem} \label{Invol_HC}
Let a reductive group $G$ act on a smooth affine variety $X$.
Suppose that for any $G$-semisimple $x \in X(F)$, the action of
$G_x$ on $N_{Gx,x}^X$ is weakly linearly tame. Then the action of
$G$ on $X$ is tame.
\end{theorem}

The proof is rather straightforward except for one minor
complication: the group of automorphisms of $X(F)$ generated by
the action of $G(F)$ is not necessarily a group of $F$-points of
any algebraic group.

\begin{proof}
Let $\tau:X \to X$ be an admissible automorphism.

Let $\widetilde{G} \subset \Aut(X)$ be the algebraic group
generated by the actions of $G$ and $\tau$. Let $K \subset
\Aut(X(F))$ be the B-analytic group generated by the action of
$G(F)$. Let $\widetilde{K} \subset \Aut(X(F))$ be the B-analytic
group generated by the actions of $G$ and $\tau$. Note that
$\widetilde{K} \subset \widetilde{G}(F)$ is an open subgroup of
finite index. Note that for any $x \in X(F)$, $x$ is
$\widetilde{G}$-semisimple if and only if it is $G$-semisimple. If
$K = \widetilde{K}$ we are done, so we will assume $K \neq
\widetilde{K}$. Let $\chi$ be the character of $\widetilde{K}$
defined by $\chi(K)=\{1\}$, $\chi(\widetilde{K} - K) = \{-1\}$.

It is enough to prove that $\Sc^*(X)^{\widetilde{K},\chi}=0$. By
Generalized Harish-Chandra Descent (Corollary \ref{Strong_HC_Cor})
it is enough to prove that for any $G$-semisimple $x \in X$ such
that $$\Sc^*(R(N_{Gx,x}^X))^{\widetilde{K}_x,\chi}=0$$ we have
$$\Sc^*(Q(N_{Gx,x}^X))^{\widetilde{K}_x,\chi}=0.$$ Choose any
automorphism $\tau'\in \widetilde{K}_x - K_x$. Note that $\tau'$
and $K_x$ generate $\widetilde{K}_x$. Denote $$\eta :=
d\tau'|_{N_{Gx,x}^X(F)}.$$ By Proposition \ref{AdmisDescends},
$\eta$ is $G_x$-admissible. Note that
$$\Sc^*(R(N_{Gx,x}^X))^{K_x}= \Sc^*(R(N_{Gx,x}^X))^{G(F)_x} \text{ and }
\Sc^*(Q(N_{Gx,x}^X))^{K_x}= \Sc^*(Q(N_{Gx,x}^X))^{G(F)_x}.$$
Hence we have $$\Sc^*(R(N_{Gx,x}^X))^{G(F)_x} \subset
\Sc^*(R(N_{Gx,x}^X))^{\eta}.$$ Since the action of $G_x$ is weakly
linearly tame, this implies that $$\Sc^*(Q(N_{Gx,x}^X))^{G(F)_x}
\subset \Sc^*(Q(N_{Gx,x}^X))^{\eta}$$ and therefore
$\Sc^*(Q(N_{Gx,x}^X))^{\widetilde{K}_x,\chi}=0$.
\end{proof}

\begin{definition} \label{DefSpec}
We call an  $F$-rational representation $V$ of a reductive group
$G$  \textbf{special} if there is no non-zero $\xi \in
\Sc^*_{Q(V)}(\Gamma(V))^{G(F)}$ such that for any $G$-invariant
decomposition $Q(V) = W_1 \oplus W_2$ and any two $G$-invariant
symmetric non-degenerate bilinear forms $B_i$ on $W_i$ the Fourier
transforms $\Fou_{B_i}(\xi)$ are also supported in $\Gamma(V)$.
\end{definition}

\begin{proposition} \label{SpecWeakTameAct}
Every special representation $V$ of a reductive group $G$ is
weakly linearly tame.
\end{proposition}


The proposition follows immediately from the following lemma.

\begin{lemma}
Let $V$ be an  $F$-rational representation of a reductive group
$G$. Let $\tau$ be an admissible linear automorphism of $V$. Let
$V = W_1 \oplus W_2$ be a $G$-invariant decomposition of $V$ and
$B_i$ be $G$-invariant symmetric non-degenerate  bilinear forms on
$W_i$. Then $W_i$ and $B_i$ are also $\tau$-invariant.
\end{lemma}

This lemma follows in turn from the following one.

\begin{lemma}
Let $V$ be an  $F$-rational representation of a reductive group
$G$. Let $\tau$ be an admissible automorphism of $V$. Then
$\cO(V)^{G} \subset \cO(V)^{\tau}$.
\end{lemma}
\begin{proof}
Consider the projection $\pi:V \to V/G$. We have to show that
$\tau$ acts trivially on $V/G$ and  let $x \in \pi(V(F))$. Let
$X:=\pi^{-1}(x)$. By Proposition \ref{LocClosedOrbit} $G(F)$ has a
closed orbit in $X(F)$. The automorphism $\tau$ preserves this
orbit and hence preserves $x$. Thus $\tau$ acts trivially on
$\pi(V(F))$, which is Zariski dense in $V/G$. Hence $\tau$ acts
trivially on $V/G$.
\end{proof}

Now we introduce a criterion that allows to prove that a
representation is special. It follows immediately from Theorem
\ref{ArchHom}.

\begin{lemma} \label{SpecActCrit}
Let $V$ be an  $F$-rational representation of a reductive group
$G$. Let $Q(V) = \bigoplus W_i$ be a $G$-invariant decomposition.
Let $B_i$ be symmetric non-degenerate $G$-invariant bilinear forms
on $W_i$. Suppose that any $\xi \in
\Sc^*_{Q(V)}(\Gamma(V))^{G(F)}$ which is adapted to all $B_i$ is
zero. Then $V$ is special.
\end{lemma}
\Rami{We will also need the following elementary fact:
\begin{lemma}\label{lem:fin.tame}
    Let $(\pi,V)$ be an  $F$-rational representation of a reductive group
$G$. Assume that $\pi(G(F))$ is finite. Then $V$ is linearly tame (and weakly linearly tame).
\end{lemma}
\begin{proof}
    Let $\tau:V\to V$ be a $G$-admissible linear automorphism. 
    It is enough to show that there is $g\in G(F)$ s.t. $\tau=\pi(g)$.

For ant $g\in G(F)$ let $V_g=\{v|\pi(g)(v)=\tau(v)\}$. The admissibility implies that $$V=\bigcup_{g\in G(F)} V_g.$$
On the other hand the assumption implies that that this union is finite. 
Since this is a union of vector spaces and $F$ is infinite, one of the spaces is $V$.
This implies the assertion.
    
\end{proof}
}
\part{Symmetric and Gelfand pairs}
\section{Symmetric pairs} \label{SecSymPairs}
In this section we apply our tools to symmetric pairs. We
introduce several properties of symmetric pairs and discuss their
interrelations. In Appendix \ref{Diag} we present a diagram that
illustrates the most important ones.

\subsection{Preliminaries and notation}

\begin{definition}
A \textbf{symmetric pair} is a triple $(G,H,\theta)$ where $H
\subset G$ are reductive groups, and $\theta$ is an involution of
$G$ such that $H = G^{\theta}$. We call a symmetric pair
\textbf{connected} if $G/H$ is connected.

For a symmetric pair $(G,H,\theta)$ we define an antiinvolution
$\sigma :G \to G$ by $$\sigma(g):=\theta(g^{-1}),$$ denote
$\g:=\Lie G$, $\h := \Lie H$. Let $\theta$ and $\sigma$ act on
$\g$ by their differentials and denote $$\gd:=\{a \in \g\ |\
\sigma(a)=a\}=\{a \in \g\ |\ \theta(a)=-a\}.$$ Note that $H$ acts
on $\gd$ by the adjoint action. Denote also $$G^{\sigma}:=\{g \in
G\ |\ \sigma(g)=g\}$$ and define a \textbf{symmetrization map}
$s:G \to G^{\sigma}$ by
$$s(g):=g \sigma(g).$$
\end{definition}

We will consider the action of $H\times H$ on $G$ by left and
right translation and the conjugation action of $H$ on $G^\sigma$.

\begin{definition}
Let $(G_1,H_1,\theta_1)$ and $(G_2,H_2,\theta_2)$ be symmetric
pairs. We define their \textbf{product} to be the symmetric pair
$(G_1 \times G_2,H_1 \times H_2,\theta_1 \times \theta_2)$.
\end{definition}

\begin{theorem} \label{RegFun}
For any connected symmetric pair $(G,H,\theta)$ we have
$\mathcal{O}(G)^{H \times H} \subset \mathcal{O}(G)^{\sigma}$.
\end{theorem}

\begin{proof}
Consider the multiplication map $H \times G^{\sigma} \to G$. It is
\et at $1 \times 1$ and hence its image $HG^{\sigma}$ contains an
open neighborhood of 1 in $G$. Hence the image of $HG^{\sigma}$ in
$G/H$ is dense. Thus $HG^{\sigma}H$ is dense in $G$. Clearly
$\mathcal{O}(H G^{\sigma} H)^{H \times H} \subset \mathcal{O}(H
G^{\sigma} H)^{\sigma}$ and hence $\mathcal{O}(G)^{H \times H}
\subset \mathcal{O}(G)^{\sigma}$.
\end{proof}
\begin{corollary} \label{ClosedOrbits}
For any connected symmetric pair $(G,H,\theta)$ and any closed $H
\times H$ orbit $\Delta \subset G$, we have
$\sigma(\Delta)=\Delta$.
\end{corollary}
\begin{proof}
Denote $\Upsilon:=H \times H$. Consider the action of the
2-element group $(1,\tau)$ on $\Upsilon$ given by $\tau(h_1,h_2):=
(\theta(h_2), \theta(h_1))$. This defines the semi-direct product
$\widetilde{\Upsilon}:= (1,\tau) \ltimes \Upsilon $. Extend the
two-sided action of $\Upsilon$ to $\widetilde{\Upsilon}$ by the
antiinvolution $\sigma$. Note that the previous theorem implies
that $G/\Upsilon = G/\widetilde{\Upsilon}$. Let $\Delta$ be a
closed $\Upsilon$-orbit. Let $\widetilde{\Delta}:=\Delta \cup
\sigma(\Delta)$. Let $a := \pi_G(\widetilde{\Delta}) \subset
G/\widetilde{\Upsilon}$. Clearly, $a$ consists of one point. On
the other hand, $G/\widetilde{\Upsilon} = G/\Upsilon$ and hence
$\pi_G^{-1}(a)$ contains a unique closed $\Rami{\Upsilon}$-orbit. Therefore
$\Delta = \widetilde{\Delta} = \sigma(\Delta)$.
\end{proof}

\begin{corollary} \label{GoodCrit}
Let $(G,H,\theta)$ be a connected symmetric pair. Let $g \in G(F)$
be $H\times H$-semisimple. Suppose that the Galois cohomology
$H^1(F,(H \times H)_g)$ is trivial. Then $\sigma(g) \in
H(F)gH(F)$.
\end{corollary}
For example, if $(H \times H)_g$ is a product of general linear
groups over some field extensions then $H^1(F,(H \times H)_g)$ is
trivial.

\begin{definition} \label{DefGoodPair}
A symmetric pair $(G,H,\theta)$ is called \textbf{good} if for any
closed $H(F) \times H(F)$ orbit $O \subset G(F)$, we have
$\sigma(O)=O$.
\end{definition}

\begin{corollary} \label{ComplexGood}
Any connected symmetric pair over $\C$ is good.
\end{corollary}

\begin{definition} \label{DefGKPair}
A symmetric pair $(G,H,\theta)$ is called a \textbf{GK-pair} if
$$\Sc^*(G(F))^{H(F) \times H(F)} \subset \Sc^*(G(F))^{\sigma}.$$
\end{definition}

We will see later in \S \ref{Gel} that GK-pairs satisfy a Gelfand
pair property that we call GP2 (see Definition \ref{GPs} and
Theorem \ref{DistCrit}). Clearly every GK-pair is good and we
conjecture that the converse is also true. We will discuss it in
more detail in \S\S \ref{conj}.

\begin{lemma} \label{BilForm}
Let $(G,H,\theta)$ be a symmetric pair. Then there exists a
$G$-invariant $\theta$-invariant non-degenerate symmetric bilinear
form $B$ on $\g$. In particular, $\g=\g^\sigma\oplus\h$ is an
orthogonal direct sum with respect to $B$.
\end{lemma}
\begin{proof}$ $\\
\indent
Step 1. Proof for semisimple $\g$.\\
Let $B$ be the Killing form on $\g$. Since it is non-degenerate,
it is enough to show that $\h$ is orthogonal to $\g^{\sigma}$. Let
$A \in \h$ and $C \in \g^{\sigma}$. We have to show
$\tr(\ad(A)\ad(C))=0$. This follows from the fact that
$\ad(A)\ad(C)(\h) \subset \g^{\sigma}$ and
$\ad(A)\ad(C)(\g^{\sigma}) \subset \h$.

Step 2. Proof in the general case.\\
Let $\g = \g' \oplus \z$ such that $\g'$ is semisimple and $\z$ is
the center. It is easy to see that this decomposition is invariant
under $Aut(\g)$ and hence $\theta$-invariant.
 Now the proposition easily follows from the
previous case.
\end{proof}
\begin{remark} \label{ExpMap}
Let $(G,H,\theta)$ be a symmetric pair. Let $\mathcal{U}(G)$ be
the set of unipotent elements in $G(F)$ and $\mathcal{N}(\g)$ the
set of nilpotent elements in $\g(F)$. Then the exponent map $exp:
\mathcal{N}(\g) \to \mathcal{U}(G)$ is $\sigma$-equivariant and
intertwines the adjoint action with conjugation.
\end{remark}
%


\begin{lemma} \label{SL2triple}
Let $(G,H,\theta)$ be a symmetric pair. Let $x \in \gd$ be a
nilpotent element. Then there exists a group homomorphism
$\phi:\SL_2 \to G$ such that $$d\phi(\begin{pmatrix}
  0 & 1 \\
  0 & 0\end{pmatrix}) = x, \quad
d\phi(\begin{pmatrix}
  0 & 0 \\
  1 & 0\end{pmatrix}) \in \gd \text{ and }\quad
\phi(\begin{pmatrix}
  t & 0 \\
  0 & t^{-1}\end{pmatrix}) \in H.$$
In particular $0 \in \overline{\Ad(H)(x)}$.
\end{lemma}
This lemma was essentially proven for $F=\C$ in \cite{KR}. The
same proof works for any $F$ and we repeat it here for the
convenience of the reader.

\begin{proof}

By the Jacobson-Morozov Theorem (see \cite[Chapter III, Theorems
17 and 10]{Jac}) we can complete $x$ to an $\sll_2$-triple
$(x_-,s,x)$. Let $s':= \frac{s+\theta(s)}{2}$. It satisfies
$[s',x]=2x$ and lies in the ideal $[x,\g]$ and hence by the
Morozov Lemma (see \cite[Chapter III, Lemma 7]{Jac}), $x$ and $s'$
can be completed to an $\sll_2$ triple $(\Rami{x_-''},s',x)$. Let
$x'_-:=\frac{\Rami{x_-''} - \theta(\Rami{x_-''})}{2}$. Note that $(x_-',s',x)$ is
also an $\sll_2$-triple. Exponentiating this $\sll_2$-triple to a
map $\SL_2 \to G$ we get the required homomorphism.
\end{proof}

\begin{notation} \label{dx}
In the notation of the previous lemma we denote $$D_t(x):=
\phi(\begin{pmatrix}
  t & 0 \\
  0 & t^{-1}\end{pmatrix})\in H \text{ and } d(x):=d\phi(\begin{pmatrix}
  1 & 0 \\
  0 & -1\end{pmatrix})\in\h.$$
These elements depend on the choice of $\phi$. However, whenever
we use this notation, nothing will depend on their choice.
\end{notation}

\subsection{Descendants of symmetric pairs}

Recall that for a symmetric pair $(G,H,\theta)$ we consider the
$H\times H$ action on $G$ by left and right translation and the
conjugation action of $H$ on $G^\sigma$.

\begin{proposition} \label{PropDescend}
Let  $(G,H,\theta)$ be a symmetric pair. Let $g \in G(F)$ be $H
\times H$-semisimple. Let $x=s(g)$. Then \\
(i) $x$ is semisimple (both as an element of $G$ and with respect
to the
$H$-action).\\
(ii) $H_x \cong (H \times H)_g$ and $(\g_x)^{\sigma} \cong N_{H g
H,g}^G$ as $H_x$-spaces.
\end{proposition}

\begin{proof}
$ $\\ \indent (i)
Clearly the image of $HgH$ in $G/H$ is closed.
Since the symmetrization map, viewed as a map from $G/H$ to $G$, is a closed embedding, it follows that the $H$-orbit of $x$ is closed. This means that $x$ is
semisimple with respect to the $H$-action. Now we have to show
that $x$ is semisimple as an element of $G$. Let $x = x_sx_u$ be
the Jordan decomposition of $x$. The uniqueness of the Jordan
decomposition implies that both $x_u$ and $x_s$ belong to
$G^{\sigma}$. To show that $x_u=1$ it is enough to show that
$\overline{\Ad(H)(x)} \ni x_s$. We will do that in several steps.
$ $\\
\indent Step 1. Proof for the case when $x_s = 1$.\\
It follows immediately from Remark \ref{ExpMap} and Lemma
\ref{SL2triple}.

Step 2. Proof for the case when
$x_s \in Z(G)$.\\
This case follows from Step 1 since conjugation acts trivially on
$Z(G)$.

Step 3. Proof in the general case.\\
Note that $x \in G_{x_s}$ and $G_{x_s}$ is $\theta$-invariant. The
statement follows from Step 2 for the group $G_{x_s}$.

(ii) The projection to the first coordinate gives rise to an isomorphism $(H
\times H)_g \cong H_x$. Let us now show that $(\g_x)^{\sigma}
\cong N_{H g H,g}^G$. First of all, $N_{H g H,g}^G \cong \g / (\h+
\Ad(g)\h).$ Let $\theta'$ be the involution of $G$ defined by
$\theta'(y)= x \theta(y) x^{-1}$. Note that $\Ad(g)\h=
\g^{\theta'}$. Fix a non-degenerate $G$-invariant symmetric
bilinear form $B$ on $\g$ as in Lemma \ref{BilForm}. Note that $B$
is also $\theta'$-invariant and hence $$(\Ad(g)\h)^{\bot} = \{a
\in \g| \theta'(a)=-a\}.$$ Now $$N_{H g H,g}^G \cong (\h+
\Ad(g)\h)^{\bot} = \h^{\bot} \cap \Ad(g)\h^{\bot} = \{a \in \g|
\theta(a)=\theta'(a)=-a\}= (\g_x)^{\sigma}.$$
\end{proof}

It is easy to see that the isomorphism $N_{H g H,g}^G \cong
(\g_x)^{\sigma}$ is independent of the choice of $B$.

\begin{definition} \label{descendant}
In the notation of the previous proposition we will say that the
pair $(G_x,H_x,\theta|_{G_x})$ is a \textbf{descendant} of
$(G,H,\theta)$.
\end{definition}

\subsection{Tame symmetric pairs} \label{SecTamePairs}

\begin{definition} \label{DefTamePairs}
We call a symmetric pair $(G,H,\theta)$\\
(i) \textbf{tame} if the action of $H\times H$ on $G$ is tame.\\
(ii) \textbf{linearly tame} if the action of $H$ on $\g^{\sigma}$
is linearly tame.\\
(iii) \textbf{weakly linearly tame} if the action of $H$ on
$\g^{\sigma}$ is weakly linearly tame.
\end{definition}

\begin{remark}
Evidently, any good tame symmetric pair is a GK-pair.
\end{remark}

The following theorem is a direct corollary of Theorem
\ref{Invol_HC}.

\begin{theorem} \label{LinDes}
Let $(G,H,\theta)$ be a symmetric pair. Suppose that all its
descendants (including itself)  are weakly linearly tame. Then
$(G,H,\theta)$ is tame and linearly tame.
\end{theorem}

\begin{definition} \label{DefSpecPair}
We call a symmetric pair $(G,H,\theta)$ \textbf{special} if
$\g^{\sigma}$ is a special representation of $H$ (see Definition
\ref{DefSpec}).
\end{definition}

The following proposition follows immediately from Proposition
\ref{SpecWeakTameAct}.
\begin{proposition} \label{SpecWeakReg}
Any special symmetric pair is weakly linearly tame.
\end{proposition}

Using Lemma \ref{BilForm} it is easy to prove the following
proposition.
\begin{proposition}
A product of special symmetric pairs is special.
\end{proposition}

Now we would like to give a criterion of speciality for symmetric
pairs. Recall the notation $d(x)$ of \ref{dx}.

\begin{proposition}[Speciality criterion] \label{SpecCrit}
Let $(G,H,\theta)$ be a symmetric pair. Suppose that for any
nilpotent $x \in \gd$ either\\
(i)
$\tr(\ad(d(x))|_{\h_x}) < \dim Q(\g^{\sigma}) $
or\\
(ii) $F$ is non-Archimedean and
$\tr(\ad(d(x))|_{\h_x}) \neq \dim Q(\g^{\sigma})$.

Then the pair $(G,H,\theta)$ is special.
\end{proposition}

For the proof we will need the following auxiliary results.

\begin{lemma} \label{EquivDef}
Let $(G,H,\theta)$ be a symmetric pair. Then $\Gamma(\gd)$ is the
set of all nilpotent elements in $Q(\gd)$.
\end{lemma}
This lemma is a direct corollary from Lemma \ref{SL2triple}.

\begin{lemma} \label{EigenInt}
Let $(G,H,\theta)$ be a symmetric pair. Let $x \in \gd$ be a
nilpotent element. Then all the eigenvalues of
$\ad(d(x))|_{\gd/[x,\h]}$ are non-positive integers.
\end{lemma}

This lemma follows from the existence of a natural surjection $
\g/[x,\g] \twoheadrightarrow \gd/[x,\h]$ (given by the
decomposition $\g = \h \oplus \gd$)
using the following straightforward lemma.

\begin{lemma}
Let $V$ be a representation of an $\sll_2$ triple $(e,h,f)$. Then
all the eigenvalues of $h|_{V/e(V)}$ are non-positive integers.
\end{lemma}

Now we are ready to prove the speciality criterion.
\begin{proof}[Proof of Proposition \ref{SpecCrit}]
We will give a proof in the case where $F$ is Archimedean. The
case of non-Archimedean $F$ is done in the same way but with less
complications.

By Lemma \ref{SpecActCrit} and the
definition of adapted it is enough to prove
$$\Sc^*_{Q(\gd)}(\Gamma(\gd))^{H(F) \times
F^{\times},(1,\chi)}=0$$ for any character $\chi$ of $F^{\times}$
of the form
$\chi(\lambda)=u(\lambda)|\lambda|^{\dim Q(\g^{\sigma})/2}$
or
$\chi(\lambda)=u(\lambda)\Rami{\lambda}|\lambda|^{\dim Q(\g^{\sigma})/2}$, where $u$ is
some unitary character.

The set $\Gamma(\gd)$ has a finite number of $H(F)$-orbits (it
follows from Lemma \ref{EquivDef} and the introduction of
\cite{KR}). Hence it is enough to show that for any $x \in
\Gamma(\gd)$ we have
$$\Sc^*(\Rami{F^\times \cdot}\Ad(H(F))x,\Sym^k(CN_{\Ad(H(F))x}^{Q(\g^{\sigma})}))^{H(F) \times F^{\times},(1,\chi)}=0 \text{ for any } k.$$
Let $K:=\{(D_t(x),t^2) | t \in F^{\times}\} \subset (H(F)\times
F^{\times})_x.$

Note that
 $$\Delta_{(H(F)\times F^{\times})_x}((D_t(x),t^2))=|\det(\Ad(D_t(x))|_{\h_x})|
 =|t|^{\tr(\ad(d(x))|_{\h_x})}.$$

By Lemma \ref{EigenInt} the eigenvalues of the action of
$(D_t(x),t^2)$ on $(\Sym^k(Q(\gd)/[x,\h]))$ are of the form $t^l$
where $l$ is a non-positive integer.

Now by Frobenius reciprocity (Theorem \ref{Frob}) we have

\begin{multline*}
\Sc^*\left(\Rami{F^\times \cdot}(H(F))x,\Sym^k(CN_{\Ad(H(F))x}^{Q(\gd)})\right)^{H(F)
\times
F^{\times},(1,\chi)}=\\
=\Sc^*\left(\{x\},\Sym^k(CN_{\Ad(H(F))x,x}^{Q(\gd)})\otimes
\Delta_{H(F)\times F^{\times}}|_{(H(F)\times F^{\times})_x} \cdot
\Delta^{-1}_{(H(F)\times F^{\times})_x} \otimes
(1,\chi)\right)^{(H(F)\times F^{\times})_x}=\\
=\left(\Sym^k(Q(\gd)/[x,\h])\otimes \Delta_{(H(F)\times
F^{\times})_x}
\otimes (1,\chi)^{-1} \otimes_{\R} \C\right)^{(H(F)\times F^{\times})_x}  \subset\\
\subset\left(\Sym^k(Q(\gd)/[x,\h])\otimes \Delta_{(H(F)\times
F^{\times})_x} \otimes (1,\chi)^{-1}\otimes_{\R} \C\right)^{K}
\end{multline*}

which is zero since all the absolute values of the eigenvalues of
the action of any $(D_t(x),t^2) \in K$ on
$$\Sym^k(Q(\gd)/[x,\h])\otimes \Delta_{(H(F)\times F^{\times})_x}
\otimes (1,\chi)^{-1}$$ are of the form $|t|^l$ where $l < 0$.
\end{proof}
\Rami{We will also need the following immediate corollary from Lemma \ref{lem:fin.tame}:
\begin{corollary} \label{cor:fin.tame}
Let $(G,H,\theta)$ be a symmetric pair. Assume that the Lie algebra of $G$ is commutative. Then $(G,H,\theta)$ is tame (and weakly linearly tame).
\end{corollary}
}
\subsection{Regular symmetric pairs} \label{SecRegPairs}
$ $\\
In this subsection we will formulate a property which is weaker
than weakly linearly tame but still enables us to prove the GK
property for good pairs.

\begin{definition}
Let $(G,H,\theta)$ be a symmetric pair. We call an element $g \in
G(F)$ \textbf{admissible} if\\
(i) $\Ad(g)$ commutes with $\theta$ (or, equivalently, $s(g)\in Z(G)$) and \\
(ii) $\Ad(g)|_{\g^{\sigma}}$ is $H$-admissible.
\end{definition}

\begin{definition} \label{DefReg}
We call a symmetric pair $(G,H,\theta)$ \textbf{regular} if for
any admissible $g \in G(F)$ such that
$\Sc^*(R(\g^{\sigma}))^{H(F)} \subset
\Sc^*(R(\g^{\sigma}))^{\Ad(g)}$ we have
$$\Sc^*(Q(\gd))^{H(F)} \subset \Sc^*(Q(\gd))^{\Ad(g)}.$$
\end{definition}

\begin{remark}
Clearly, every weakly linearly tame pair is regular.
\end{remark}

\begin{proposition}
Let $(G_1,H_1,\theta_1)$ and $(G_2,H_2,\theta_2)$ be regular
symmetric pairs. Then their product $(G_1\times G_2,H_1 \times H_2
,\theta_1 \times \theta_2 )$ is also a regular pair.
\end{proposition}
\begin{proof}
This follows from Proposition \ref{Product}, since a product of
admissible elements is admissible, and $R(\g_1^{\sigma_2}) \times
R(\g_2^{\sigma_2})$ is an open saturated subset of $R((\g_1 \times
\g_2)^{\sigma_1 \times \sigma_2})$.
\end{proof}

The goal of this subsection is to prove the following theorem.

\begin{theorem} \label{GoodHerRegGK}
Let $(G,H,\theta)$ be a good symmetric pair such that all its
descendants are regular. Then it is a GK-pair.
\end{theorem}

We will need several definitions and lemmas.

\begin{definition}
Let $(G,H,\theta)$ be a symmetric pair. An element $g \in G$ is
called \textbf{normal} if $g$ commutes with $\sigma(g)$.
\end{definition}

Note that if $g$ is normal then
$$g\sigma(g)^{-1}=\sigma(g)^{-1}g\in H.$$

The following lemma is straightforward.
\begin{lemma}
Let $(G,H,\theta)$ be a symmetric pair. Then any
$\sigma$-invariant $H(F) \times H(F)$-orbit in $G(F)$ contains a
normal element.
\end{lemma}

\begin{proof} $ $\\
Let $g' \in O$. We know that $\sigma(g') = h_1 g' h_2$ where
$h_1,h_2 \in H(F)$. Let $g:=g' h_1$. Then
\begin{multline*}
\sigma(g)g = h_1^{-1} \sigma(g') g' h_1 = h_1^{-1} \sigma(g')
\sigma(\sigma(g')) h_1 = \\ = h_1^{-1} h_1 g' h_2
\sigma(h_1g'h_2))h_1= g' \sigma(g') = g' h_1 h_1^{-1} \sigma(g')=
g \sigma(g).
\end{multline*}
Thus $g$ in $O$ is normal.
\end{proof}

\begin{notation}
Let $(G,H,\theta)$ be a symmetric pair. We denote
$$\widetilde{H\times H}:= H\times H \rtimes \{1, \sigma\}$$ where
$$\sigma \cdot (h_1,h_2) = (\theta(h_2),\theta(h_1)) \cdot \sigma.$$
The two-sided action of $H\times H$ on $G$ is extended to an
action of $\widetilde{H\times H}$ in the natural way. We denote by
$\chi$ the character of $\widetilde{H\times H}$ defined by
$$\chi(\widetilde{H\times H} - H \times H)=\{-1\},\quad \chi(H \times
H)=\{1\}.$$
\end{notation}

\begin{proposition}
Let $(G,H,\theta)$ be a good symmetric pair. Let $O \subset G(F)$
be a closed $H(F) \times H(F)$-orbit.

Then for any $g \in O$ there exist $\tau \in (\widetilde{H\times
H})_g(F) - (H\times H)_g(F)$ and $g' \in G_{s(g)}(F)$ such that
$\Ad(g')$ commutes with $\theta$ on $G_{s(g)}$ and the action of
$\tau$ on $N_{O,g}^G$ corresponds via the isomorphism given by
Proposition \ref{PropDescend} to the adjoint action of $g'$ on
$\g_{s(g)}^{\sigma}$.
\end{proposition}
\begin{proof}
Clearly, if the statement holds for some $g\in O$ then it holds
for all $g\in O$.

Let $g \in O$ be a normal element. Let $h:=\Rami{g^{-1}\sigma(g)}$.
Recall that $h\in H(F)$ and $gh=hg=\sigma(g)$. Let $\tau:=
(h^{-1},1) \cdot \sigma$. Evidently, $\tau\in (\widetilde{H\times
H})_g(F) - (H\times H)_g(F)$. Consider $d\tau_g: T_gG \to T_gG$.
It corresponds via the identification $dg: \g \cong T_gG$ to some
$A:\g \to \g$. Clearly, $A = da$ where $a:G \to G$ is defined by
$a(\alpha) = g^{-1}h^{-1}\sigma(g\alpha)$. However,
$g^{-1}h^{-1}\sigma(g\alpha) = \theta(g) \sigma(\alpha)
\theta(g)^{-1}.$ Hence $A = \Ad(\theta(g)) \circ \sigma$. By Lemma
\ref{BilForm}, there exists a non-degenerate $G$-invariant
$\sigma$-invariant symmetric bilinear form $B$ on $\g$. By Theorem
\ref{RegFun}, $A$ preserves $B$. Therefore $\tau$ corresponds to
$A|_{\g_{s(g)}^{\sigma}}$ via the isomorphism given by Proposition
\ref{PropDescend}. However, $\sigma$ is trivial on
$\g_{s(g)}^{\sigma}$ and hence
$A|_{\g_{s(g)}^{\sigma}}=\Ad(\theta(g))|_{\g_{s(g)}^{\sigma}}$.
Since $g$ is normal, $\theta(g) \in G_{s(g)}$. It is easy to see
that $\Ad(\theta(g))$ commutes with $\theta$ on $G_{s(g)}$. Hence
we take $g':=\theta(g)$.
\end{proof}

Now we are ready to prove Theorem \ref{GoodHerRegGK}.
\begin{proof}[Proof of Theorem \ref{GoodHerRegGK}]
We have to show that $\Sc^*(G(F))^{\widetilde{H\times H},
\chi}=0$. By \Rami{Corollary} \ref{Strong_HC_Cor} it is enough to show that
for any $H \times H$-semisimple $x \in G(F)$ such that
$$\cD(R(N_{HxH,x}^G))^{\widetilde{(H(F)\times H(F))}_x, \chi}=0$$ we have
$$\cD(Q(N_{HxH,x}^G))^{\widetilde{(H(F)\times H(F))}_x, \chi}=0.$$
This follows immediately from the regularity of the pair
$(G_x,H_x)$ using the last proposition \Rami{and Theorem \ref{NoSNoDist}}.
\end{proof}


\subsection{Conjectures} \label{conj}
\begin{conjecture}[van Dijk] \label{ConjDijk}
If $F=\C$ then any connected symmetric pair is a Gelfand pair
(GP3, see Definition \ref{GPs} below).
\end{conjecture}
By Theorem \ref{DistCrit} this would follow from the following
stronger conjecture.
\begin{conjecture}\label{ConjDijkGK}
If $F=\C$ then any connected symmetric pair is a GK-pair.
\end{conjecture}
By Corollary \ref{ComplexGood} this in turn would follow from the
following more general conjecture.
\begin{conjecture} \label{ConjGKGood}
Every good symmetric pair is a GK-pair.
\end{conjecture}
which in turn follows (by Theorem \ref{GoodHerRegGK})  from the
following one.
\begin{conjecture} \label{ConjAllReg}
Any symmetric pair is regular.
\end{conjecture}
\begin{remark}
In the next two subsections we prove this conjecture for certain
symmetric pairs. In subsequent works
\cite{AG_RegSymPairs,Say1,AS,Say2,Aiz} this conjecture was
verified for most classical symmetric pairs and several
exceptional ones.
\end{remark}

\begin{remark}
An indirect evidence for this conjecture is that every GK-pair is
regular. One can easily show this by analyzing a Luna slice for an
orbit of an admissible element.
\end{remark}
\begin{remark}
It is well known that if $F$ is Archimedean, $G$ is connected and
$H$ is compact then the pair $(G,H,\theta)$ is good, Gelfand (GP1,
see Definition \ref{GPs} below) and in fact also GK. See e.g.
\cite{Yak}.
\end{remark}

\begin{remark}
In general, not every symmetric pair is good. For example,
$(\SL_2(\R),T)$ where $T$ is the split torus. Also, it is not a
Gelfand pair (not even  GP3, see Definition \ref{GPs} below).
\end{remark}

\begin{remark}
It seems unlikely that every symmetric pair is special. However,
in the next two subsections we will prove that certain symmetric
pairs are special.
\end{remark}

\subsection{The pairs $(G \times G, \Delta G)$ and $(G_{E/F}, G)$
are tame} \label{Sec2RegPairs}

\begin{notation}
Let $E$ be a quadratic extension of $F$. Let $G$ be an algebraic
group defined over $F$. We denote by $G_{E/F}$ the restriction of
scalars from $E$ to $F$ of $G$ viewed as a group over $E$. Thus,
$G_{E/F}$ is an algebraic group defined over $F$ and
$G_{E/F}(F)=G(E)$.
\end{notation}

In this section we will prove the following theorem.

\begin{theorem}\label{2RegPairs}
Let $G$ be a reductive group.\\
(i) Consider the involution $\theta$ of $G\times G$ given by
$\theta((g,h)):= (h,g)$. Its fixed points form the diagonal
subgroup $\Delta G$. Then the symmetric pair $(G \times G, \Delta
G, \theta)$ is
tame.\\
(ii) Let $E$ be a quadratic extension of $F$. Consider the
involution $\gamma$ of $G_{E/F}$ given by the nontrivial element
of $Gal(E/F)$. Its fixed points form $G$. Then the symmetric pair
$(G_{E/F}, G, \gamma)$ is tame.
\end{theorem}

\begin{corollary} \label{GCongTame}
Let $G$ be a reductive group. Then the adjoint action of $G$ on
itself is tame. In particular, every conjugation invariant
distribution on $\GL_n(F)$ is transposition invariant \footnote{In
the non-Archimedean case, the latter is a classical result of
Gelfand and Kazhdan, see \cite{GK}.}.
\end{corollary}

For the proof of the theorem we will need the following
straightforward lemma.
\begin{lemma} $ $\\
(i) Every descendant of $(G \times G, \Delta G, \theta)$ is of the
form $(H \times H, \Delta H, \theta)$ for some reductive group
$H$.

\noindent (ii) Every descendant of $(G_{E/F}, G, \gamma)$ is of
the form $(H_{E/F}, H, \gamma)$ for some reductive group $H$.
\end{lemma}

Now in view of \Rami{Theorems \ref{SpecWeakReg} and \ref{LinDes}, and Corollary \ref{cor:fin.tame}}, Theorem \ref{2RegPairs}
follows from the following theorem.

\begin{theorem}
\Rami{Let $G$ be a reductive group with a non-commutative Lie algebra. Then
t}he pairs $(G \times G, \Delta G, \theta)$ and  $(G_{E/F}, G,
\gamma)$ are special.
\end{theorem}

By the speciality criterion (Proposition \ref{SpecCrit}) this
theorem follows from the following lemma.

\begin{lemma}
Let $\g$ be a \Rami{non-trivial} semisimple Lie algebra. Let $\{e,h,f\} \subset \g$
be an $\sll_2$ triple. Then $\tr(\ad(h)|_{{\g}_e})$ is an integer
smaller than $\dim\g$.
\end{lemma}
\begin{proof}
Consider $\g$ as a representation of $\sll_2$ via the triple
$(e,h,f)$. Decompose it into irreducible representations
$\g=\bigoplus V_i$. Let $\lambda_i$ be the highest weights of
$V_i$. Clearly $$\tr(\ad(h)|_{{\g}_e})= \sum \lambda_i \text{
while } \dim\g= \sum (\lambda_i+1).$$
\end{proof}

\subsection{The pair  $(\mathrm{GL}_{n+k},\mathrm{GL}_{n}\times \mathrm{GL}_{k})$ is a GK
pair.} \label{RJR}
\begin{notation}
We define an involution $\theta_{n,k}:\GL_{n+k} \to \GL_{n+k}$ by
$\theta_{n,k}(x)=\eps x \eps$ where $\eps = \begin{pmatrix}
  I_{n} & 0 \\
  0 & -I_{k}
\end{pmatrix}$.
%
Note that
$(\GL_{n+k},\GL_{n} \times \GL_{k},\theta_{n,k})$ is a symmetric
pair. If there is no ambiguity we will denote $\theta_{n,k}$
simply by $\theta$.
\end{notation}

\begin{theorem} \label{thm2}
The pair $(\GL_{n+k},\GL_{n}\times \GL_{k},\theta_{n,k})$ is a
GK-pair.
\end{theorem}

By Theorem \ref{GoodHerRegGK} it is enough to prove that our pair
is good and all its descendants are regular.

In \S\S\S \ref{first} we  compute the descendants of our pair and
show that the pair is good.

In \S\S\S \ref{second} we  prove that all the descendants are
regular. 

\subsubsection{The descendants of the
pair $(\GL_{n+k},\GL_{n} \times \GL_k)$} \label{first}

\begin{theorem} \label{ComputeDescent}
All the descendants of the pair $(\GL_{n+k},\GL_{n}\times
\GL_{k},\theta_{n,k})$ are products of pairs of the types

(i) $((\GL_m)_{E/F} \times (\GL_m)_{E/F}, \Delta (\GL_m)_{E/F},
\theta)$ for some field extension $E/F$

(ii) $((\GL_m)_{E/F}, (\GL_m)_{L/F}, \gamma)$ for some field
extension $L/F$ and its quadratic extension $E/L$

(iii) $(\GL_{m+l},\GL_{m}\times \GL_{l},\theta_{m,l})$.
\end{theorem}
\begin{proof}
Let $x \in \GL_{n+k}^{\sigma}(F)$ be a semisimple element. We have
to compute $G_x$ and $H_x$. Since $x \in G^{\sigma}$, we have
$\eps x \eps = x^{-1}$. Let $V = F^{n+k}$. Decompose $V :=
\bigoplus _{i=1}^s V_i$ such that the minimal polynomial of
$x|_{V_i}$ is irreducible. Now $G_x(F)$ decomposes as a product of
$\GL_{E_i}(V_i)$, where $E_i$ is the extension of $F$ defined by
the minimal polynomial of $x|_{V_i}$ and the $E_i$-vector space
structure on $V_i$ is given by $x$.

Clearly, $\eps$ permutes the $V_i$'s.
Now we see that $V$ is a direct sum of spaces of the following two types\\
A. $W_1 \oplus W_2$ such that the minimal polynomials of
$x|_{W_i}$ are irreducible and $\eps(W_1)=W_2$.\\
B. $W$ such that the minimal polynomial of $x|_{W}$ is irreducible
and $\eps(W)=W$.

It is easy to see that in case A we get the symmetric pair (i).

In case B there are two possibilities: either  $x = x^{-1}$ or  $x
\neq x^{-1}$. It is easy to see that these cases correspond to
types (iii) and (ii) respectively.
\end{proof}


\begin{corollary}
The pair ($\GL_{n+k},\GL_{n} \times \GL_k$) is good.
\end{corollary}
\begin{proof}
Theorem \ref{ComputeDescent} implies that for any $(\GL_{n} \times
\GL_k) \times (\GL_{n} \times \GL_k)$-semisimple element $x \in
\GL_{n+k}(F)$, the stabilizer $((\GL_{n} \times \GL_k) \times
(\GL_{n} \times \GL_k))_x$ is a product of groups of types
$(\GL_m)_{E/F}$ for some extensions $E/F$. Hence $H^1(F,((\GL_{n}
\times \GL_k) \times (\GL_{n} \times \GL_k))_x)=0$ and hence by
Corollary \ref{GoodCrit} the pair ($\GL_{n+k},\GL_{n} \times
\GL_k$) is good.
\end{proof}

\subsubsection{All the descendants of the pair $(\GL_{n+k},\GL_{n}
\times \GL_k)$ are regular} \label{second}
$ $\\
Clearly, for any field extension $E/F$, if a pair $(G,H,\theta)$
is regular as a symmetric pair over $E$ then the pair
$(G_{E/F},H_{E/F},\theta)$ is  regular. Therefore by Theorem
\ref{ComputeDescent} and \Rami{\S\S\ref{Sec2RegPairs}}
it is enough to
prove that the pair $(\GL_{n+k},\GL_{n}\times
\GL_{k},\theta_{n,k})$ is regular as a symmetric pair over $F$.

In the case $n \neq k $ this follows from the definition since in
this case the normalizer of $\GL_n\times\GL_k$ in $GL_{k+n}$ is
$\GL_n\times\GL_k$ and hence, any admissible $g\in \GL_{n+k}$ lies
in $\GL_n \times \GL_k$.

So we can assume $n=k>0$. Hence by Proposition \ref{SpecCrit} it
suffices to prove the following Key Lemma.

\begin{lemma}[Key Lemma]\footnote{This Lemma is similar to \cite[\S\S 3.2, Lemma 3.1]{JR}. The
proofs are also similar.} \label{Key} Let $x \in
\gl_{2n}^{\sigma}(F)$ be a nilpotent element and $d:=d(x)$. Then
$$\tr(\ad(d)|_{(\gl_{n}(F)\times \gl_n(F))_x}) < 2n^2.$$
\end{lemma}

We will need the following definition and lemmas.

\begin{definition}
We fix a grading on $\sll_2(F)$ given by $h \in \sll_2(F)_0$ and
$e,f \in \sll_2(F)_1$ where $(e,h,f)$ is the standard
$\sll_2$-triple. A \textbf{graded representation of $\sll_2$} is a
representation of $\sll_2$ on a graded vector space $V=V_0 \oplus
V_1$ such that $\sll_2(F)_i(V_j) \subset V_{i+j}$ where $i,j \in
\Z/2\Z$.
\end{definition}

The following lemma is standard.
\begin{lemma}$ $\\
(i) Every irreducible graded representation of $\sll_2$ is
irreducible (as a usual representation of $\sll_2$).

\noindent (ii) Every irreducible representation $V$ of $\sll_2$
admits exactly two gradings. In one grading the highest weight
vector lies in $V_0$ and in the other grading it lies in $V_1$.
\end{lemma}

\begin{notation}
Denote by $V_{\lambda}^w$ be the irreducible graded representation
of $\sll_2$ with highest weight $\lambda$ and highest weight
vector of parity $w\in \Z / 2 \Z$.
\end{notation}

\begin{lemma} \label{mij}
\footnote{This Lemma is similar to \cite[Lemma 3.2]{JR} but computes a
different quantity.} Consider
$\Hom((V_{\lambda_1}^{w_1},V_{\lambda_{2}}^{w_2})^e)_0$ - the even
part of the space of $e$-equivariant linear maps
$V_{\lambda_1}^{w_1} \to V_{\lambda_{2}}^{w_2}$. Let $r_i:=\dim
V_{\lambda_i}^{w_i}= \lambda_i +1$ and let   $$m:=
\tr(h|_{(\Hom((V_{\lambda_1}^{w_1},V_{\lambda_{2}}^{w_2})^e)_0}) +
\tr(h|_{\Hom((V_{\lambda_2}^{w_2},V_{\lambda_{1}}^{w_1})^e)_0}) -
r_1r_2.$$ Then
$$m=\left\{%
\begin{array}{lll}
    -\min(r_1,r_2), &\text{if}\quad r_1 \neq r_2 & \pmod 2; \\
    -2\min(r_1,r_2), &\text{if}\quad r_1 \equiv r_2 \equiv 0 & \pmod 2 \text{ and }w_1=w_2 ; \\
 0, &\text{if}\quad r_1 \equiv r_2 \equiv 0 & \pmod 2  \text{ and }w_1 \neq w_2; \\
    |r_1-r_2|-1, &\text{if}\quad  r_1 \equiv r_2 \equiv 1 & \pmod 2  \text{ and }w_1 = w_2; \\
-(r_1+r_2-1), &\text{if}\quad  r_1 \equiv r_2 \equiv 1 &  \pmod 2  \text{ and }w_1 \neq w_2; \\
\end{array}%
\right.$$
\end{lemma}
This lemma follows by a direct computation from the following
straightforward lemma.
\begin{lemma}
One has \setcounter{equation}{0}
\begin{align}
& \tr(h|_{((V_{\lambda}^w)^e)_0}) = \left\{%
\begin{array}{ll}
    \lambda, &\text{if } w=0 \\
    0, &\text{if } w=1
\end{array}%
\right.\\
& (V_{\lambda}^w)^* = V_{\lambda}^{w+{\lambda}}\\
& V_{\lambda_1}^{w_1} \otimes V_{\lambda_2}^{w_2} =
\bigoplus_{i=0}^{\min(\lambda_1, \lambda_2)}
V_{\lambda_1+\lambda_2 - 2i}^{w_1+w_2+i}.
\end{align}
\end{lemma}

\begin{proof}[Proof of the Key Lemma]
Let $V_0:=V_1:=F^n$. Let $V:=V_0 \oplus V_1$ be a
$\Z/{2\Z}$-graded vector space. We consider $\gl_{2n}(F)$ as the
$\Z/{2\Z}$-graded Lie algebra $\End(V)$. Note that $\gl_n(F)
\times \gl_n(F)$ is the even part of $\End(V)$ with respect to
this grading. Consider $V$ as a graded representation of the
$\sll_2$ triple $(x, d,x_{-})$. Decompose $V$ into graded
irreducible representations $W_i$. Let $r_i:=\dim W_i$ and $w_i$
be the parity of the highest weight vector of $W_i$. Note that if
$r_i$ is even then $\dim (W_i \cap V_0)=\dim (W_i \cap V_1)$. If
$r_i$ is odd then $\dim (W_i \cap V_0)=\dim (W_i \cap V_1)
+(-1)^{w_i}$. Since $\dim V_0 = \dim V_1$, we get that the number
of indices $i$ such that $r_i$ is odd and $w_i=0$ is equal to the
number of indices $i$ such that $r_i$ is odd and $w_i=1$. We
denote this number by $l$. Now
$$\tr(\ad(d)|_{(\gl_{n}(F) \times \gl_n(F))_x}) - 2n^2 =
\tr(d|_{(\Hom(V,V)^x)_0})-2n^2= \frac{1}{2}\sum_{i,j} m_{ij},$$
where
$$m_{ij}:=\tr(d|_{(\Hom(W_i,W_j)^x)_0}) + \tr(d|_{(\Hom(W_j,W_i)^x)_0}) -
r_ir_j.$$ The $m_{ij}$ can be computed using Lemma \ref{mij}.

As we see from the lemma, if either $r_i$ or $r_j$ is even then
$m_{ij}$ is non-positive and $m_{ii}$ is negative. Therefore, if
all $r_i$ are even then we are done. Otherwise $l>0$ and we can
assume that all $r_i$ are odd. Reorder the spaces $W_i$ so that
$w_i = 0$ for $i \leq l$ and $w_i=1$ for $i>l$. Now

\begin{multline*}
\sum_{1 \leq i,j\leq 2l}m_{ij} = \sum_{i \leq l, j \leq l}
(|r_i-r_j|-1) + \sum_{i
> l, j > l} (|r_i-r_j|-1) - \sum_{i \leq l, j > l} (r_i+r_j-1) - \sum_{i > l, j \leq l}
(r_i+r_j-1)=\\
=\sum_{i \leq l, j \leq l} |r_i-r_j| + \sum_{i
> l, j > l} |r_i-r_j| - \sum_{i \leq l, j > l} (r_i+r_j) - \sum_{i > l, j \leq l}
(r_i+r_j) <\\
<\sum_{i \leq l, j \leq l} (r_i+r_j) + \sum_{i
> l, j > l} (r_i+r_j) - \sum_{i \leq l, j > l} (r_i+r_j) - \sum_{i > l, j \leq l}
(r_i+r_j)=0.
\end{multline*}
The Lemma follows.
\end{proof}

 \section{Applications to Gelfand pairs} \label{Gel}
 \subsection{Preliminaries on Gelfand pairs and distributional criteria}
 $ $\\
 In this section we recall a technique due to Gelfand-Kazhdan
 which allows to deduce statements in representation theory from
 statements on invariant distributions. For more detailed
 description see \cite[\S 2]{AGS1}.

 \begin{definition}
 Let $G$ be a reductive group. By an \textbf{admissible
 representation of} $G$ we mean an admissible representation of
 $G(F)$ if $F$ is non-Archimedean (see \cite{BZ}) and admissible
 smooth \Fre representation of $G(F)$ if $F$ is Archimedean.
 \end{definition}

 We now introduce three a-priori distinct notions of Gelfand pair.

 \begin{definition}\label{GPs}
 Let $H \subset G$ be a pair of reductive groups.
 \begin{itemize}
 \item We say that $(G,H)$ satisfy {\bf GP1} if for any irreducible
 admissible representation $(\pi,E)$ of $G$
 we have
 $$\dim \Hom_{H(F)}(E,\cc) \leq 1.$$

 \item We say that $(G,H)$ satisfy {\bf GP2} if for any irreducible
 admissible representation $(\pi,E)$ of $G$
 we have
 $$\dim \Hom_{H(F)}(E,\cc) \cdot \dim \Hom_{H}(\widetilde{E},\cc)\leq
 1.$$

 \item We say that $(G,H)$ satisfy {\bf GP3} if for any irreducible
 {\bf unitary} representation $(\pi,\mathcal{H})$ of $G(F)$ on a
 Hilbert space $\mathcal{H}$ we have
 $$\dim \Hom_{H(F)}(\mathcal{H}^{\infty},\cc) \leq 1.$$
 \end{itemize}

 \end{definition}

 Property GP1 was established by Gelfand and Kazhdan in certain
 $p$-adic cases (see \cite{GK}). Property GP2 was introduced in
 \cite{Gross} in the $p$-adic setting. Property GP3 was studied
 extensively by various authors under the name {\bf generalized
 Gelfand pair} both in the real and $p$-adic settings (see e.g.
 \cite{vD-P}, \cite{vD}, \cite{Bos-vD}).

 We have the following straightforward proposition.

 \begin{proposition}
 $GP1 \Rightarrow GP2 \Rightarrow GP3.$
 \end{proposition}

 \begin{remark}
 It is not known whether some of these notions are equivalent.
 \end{remark}

 We will use the following theorem from \cite{AGS1} which is a
 version of a classical theorem of Gelfand and Kazhdan (see
 \cite{GK}).

 \begin{theorem}\label{DistCrit}
 Let $H \subset G$ be reductive groups and let $\tau$ be an
 involutive anti-automorphism of $G$ and assume that $\tau(H)=H$.
 Suppose $\tau(\xi)=\xi$ for all bi $H(F)$-invariant Schwartz
 distributions $\xi$ on $G(F)$. Then $(G,H)$ satisfies GP2.
 \end{theorem}

 \begin{corollary} \label{GKGP2}
 Any symmetric GK-pair satisfies GP2.
 \end{corollary}

 In some cases, GP2 is known to be equivalent to GP1. For example, see
 Corollary \ref{GKCor} below.

 \subsection{Applications to Gelfand pairs}

 \begin{theorem}\label{GKRep}
 Let $G$ be a reductive group and let $\sigma$ be an
 $\Ad(G)$-admissible anti-automorphism of $G$. Let $\theta$ be the
 automorphism of $G$ defined by $\theta(g):=\sigma(g^{-1})$. Let
 $(\pi, E)$ be an irreducible admissible representation of $G$.

 Then $\widetilde{E} \cong E^{\theta}$, where $\widetilde{E}$
 denotes the smooth contragredient representation and $E^{\theta}$
 is $E$ twisted by $\theta$.
 \end{theorem}
 \begin{proof}
 By Corollary \ref{GCongTame}, the characters of $\widetilde{E}$ and
 $E^{\theta}$ are identical. Since these
 representations are irreducible, this implies that they are
 isomorphic (see e.g. \cite[Theorem 8.1.5]{Wal1}).
 \end{proof}

 \begin{remark}
 This theorem has an alternative proof using Harish-Chandra's
 Regularity Theorem, which says that the character of an admissible
 representation is a locally integrable function.
 \end{remark}

 \begin{corollary} \label{GKCor}
 Let $H \subset G$ be reductive groups and let $\tau$ be an
 $\Ad(G)$-admissible anti-automorphism of $G$ such that
 $\tau(H)=H$. Then $GP1$ is equivalent to $GP2$ for the pair
 $(G,H)$.
 \end{corollary}

This corollary, together with Corollary \ref{GKGP2} and Theorem
\ref{thm2} implies the following result.
\begin{theorem} \label{thm3}
The pair $(\GL_{n+k},\GL_{n}\times \GL_{k})$ satisfies GP1.
\end{theorem}
For non-Archimedean $F$ this theorem is proven in \cite{JR}.

 \begin{theorem}\label{Flicker}
 Let $E$ be a quadratic extension of $F$. Then the pair $((\GL_n)_{E/F},
 \GL_n)$ satisfies GP1.
 \end{theorem}
 For non-Archimedean $F$ this theorem is proven in \cite{Fli}.
 \begin{proof}
 By Theorem \ref{2RegPairs} this pair is tame. Hence it is enough
 to show that this symmetric pair is good. 
 Let $x \in \GL_n(E)^{\sigma}$ be semisimple. The
 stabilizer $(\GL_n)_x$ is a product of groups of the form
 $(\GL_n)_{F'/F}$ for some extensions $F'/F$. Hence
 $H^1(F,(\GL_n)_x)=0$. Therefore, by Corollary \ref{GoodCrit}, the
symmetric pair in question is good.
 \end{proof}

\part{Appendices}
\appendix

\section{Algebraic geometry over local fields} \label{AppLocField}
\subsection{Implicit Function Theorems}
\label{AppSub}

\begin{definition}
An analytic map $\phi:M \to N$ is called \textbf{\et} if $d_x\phi:
T_xM \to T_{\phi(x)}N$ is an isomorphism for any $x\in M$. An
analytic map $\phi:M \to N$ is called a \textbf{submersion} if
$d_x\phi: T_xM \to T_{\phi(x)}N$ is onto for any $x\in M$.
\end{definition}

We will use the following version of the Inverse Function Theorem.

\begin{theorem}[cf. \cite{Ser}, Theorem 2 in \S 9 of Chapter III in part
II] \label{InvFunct} Let $\phi: M \to N$ be an \et map of analytic
manifolds. Then it is locally an isomorphism.
\end{theorem}

\begin{corollary} \label{EtLocIs}
Let $\phi: X \to Y$ be a morphism of (not necessarily smooth)
algebraic varieties. Suppose that $\phi$ is \et at $x \in X(F)$.
Then there exists an open neighborhood $U \subset X(F)$ of $x$
such that $\phi|_U$ is a homeomorphism to its open image in
$Y(F)$.
\end{corollary}
For the proof see e.g. \cite[Chapter III, \S 5, proof of Corolary
2]{Mum}. There, the proof is given for the case $F=\C$ but it
works in general.

\begin{remark}
If $F$ is Archimedean then one can choose $U$ to be
semi-algebraic.
\end{remark}

The following proposition is well known (see e.g. \S 10 of Chapter
III in part II of \cite{Ser}).
\begin{proposition}
Any submersion $\phi:M \to N$ is open.
\end{proposition}

\begin{corollary}
Lemma \ref{OrbitIsOpen} holds. Namely, for any algebraic group $G$
and a closed algebraic subgroup $H \subset G$ the subset
$G(F)/H(F)$ is open and closed in $(G/H)(F)$.
\end{corollary}

\begin{proof}
Consider the map $\phi: G(F) \to (G/H)(F)$ defined by
$\phi(g)=gH$. Clearly, it is a submersion and its image is exactly
$G(F)/H(F)$. Hence, $G(F)/H(F)$ is open. Since each $G(F)$-orbit
in $(G/H)(F)$ is open for the same reason, $G(F)/H(F)$ is also
closed.
\end{proof}

\subsection{The Luna Slice Theorem} \label{AppLun}
$ $\\
In this subsection we formulate the Luna Slice Theorem and show
how it implies Theorem \ref{LocLuna}. For a survey on the Luna
Slice Theorem we refer the reader to \cite{Dre} and the original
paper \cite{Lun}.

\begin{definition}[cf. \cite{Dre}]
Let a reductive group $G$ act on affine varieties $X$ and $Y$. A
$G$-equivariant algebraic map $\phi:X \to Y$ is called
\textbf{strongly \et} if\\
(i) $\phi/G:X/G \to Y/G$ is \et\\
(ii) $\phi$ and the quotient morphism $\pi_X: X \to X/G$ induce a
$G$-isomorphism $X \cong Y \times _{Y/G}X/G$.
\end{definition}

\begin{definition}
Let $G$ be a reductive  group and $H$ be a closed reductive
subgroup. Suppose that $H$ acts on an affine variety $X$. Then $G
\times _{H} X$ denotes $(G\times X)/H$ with respect to the action
$h(g,x)=(gh^{-1},hx)$.
\end{definition}

\begin{theorem}[Luna Slice Theorem] \label{Luna}
Let a reductive group $G$ act on a smooth affine variety $X$. Let
$x \in X$ be $G$-semisimple.

Then there exists a locally closed smooth affine $G_x$-invariant
subvariety $Z \ni x$ of $X$ and a strongly \et algebraic map of
$G_x$ spaces $\nu: Z \to N_{Gx,x}^X$ such that the $G$-morphism
$\phi : G \times_{G_x} Z \to X$ induced by the action of $G$ on
$X$ is strongly \et.
\end{theorem}
\begin{proof}
It follows from \cite[Proposition 4.18, Lemma 5.1 and Theorems 5.2
and 5.3]{Dre}, noting that one can choose $Z$ and $\nu$ (in our
notation) to be defined over $F$.
\end{proof}

\begin{corollary}
Theorem \ref{LocLuna} holds. Namely:\\
Let a reductive group $G$ act on a smooth affine variety $X$. Let
$x \in X(F)$ be $G$-semisimple.

Then there exist\\
(i) an open $G(F)$-invariant B-analytic neighborhood $U$ of
$G(F)x$ in $X(F)$ with a
$G$-equivariant B-analytic retract $p:U \to G(F)x$ and\\
(ii) a $G_x$-equivariant B-analytic embedding $\psi:p^{-1}(x)
\hookrightarrow N_{Gx,x}^{X}(F)$ with an open saturated image such
that $\psi(x)=0$.
\end{corollary}
\begin{proof}
Let $Z$, $\phi$ and $\nu$ be as in the last theorem.

Let $Z':= Z/G_x \cong (G \times_{G_x} Z)/G$ and $X':=X/G$.
Consider the natural map $\phi': Z'(F) \to X'(F)$. By Corollary
\ref{EtLocIs} there exists a neighborhood $S' \subset Z'(F)$ of
$\pi_Z(x)$ such that $\phi'|_{S'}$ is a homeomorphism to its open
image.

Consider the natural map $\nu': Z'(F) \to N_{Gx,x}^X/G_x(F)$. Let
$S'' \subset \Rami{Z'}(F)$ be a neighborhood of $\pi_Z(x)$ such that
$\nu'|_{S''}$ is an isomorphism to its open image. In case that F
is Archimedean we choose $S'$ and $S''$ to be semi-algebraic.

Let $S:=\pi_Z^{-1}(S''\cap S') \cap Z(F)$. Clearly, $S$ is
B-analytic.

Let $\rho: (G \times_{G_x}Z)(F) \to Z'(F)$ be the natural
projection. Let $O= \rho^{-1}(S'' \cap S')$. Let $q:O \to
(G/G_x)(F)$ be the natural map. Let $O':=q^{-1}(G(F)/G_x(F))$ and
$q':=q|_{O'}$.

Now put $U := \phi(O')$ and put $p:U \to G(F)x$ be the morphism
that corresponds to $q'$. Note that $p^{-1}(x) \cong S$ and put
$\psi:p^{-1}(x) \to N_{Gx,x}^X(F)$ to be the imbedding that
corresponds to $\nu|_S$.
\end{proof}

\section{Schwartz distributions on Nash manifolds} \label{AppSubFrob}

\subsection{Preliminaries and notation}
$ $\\
In this appendix we will prove some properties of $K$-equivariant
Schwartz distributions on Nash manifolds. We work in the notation
of \cite{AG1}, where one can read about Nash manifolds and
Schwartz distributions over them. More detailed references on Nash
manifolds are \cite{BCR} and \cite{Shi}.

Nash manifolds are equipped with the \textbf{restricted topology},
in which open sets are open semi-algebraic sets. This is not a
topology in the usual sense of the word as infinite unions of open
sets are not necessarily open sets in the restricted topology.
However, finite unions of open sets are open
 and therefore in the restricted topology we consider only
finite covers. In particular, if $E \to M$ is a Nash vector bundle
it means that there exists a \underline{finite} open cover $U_i$
of $M$ such that $E|_{U_i}$ is trivial.

\begin{notation}
Let $M$ be a Nash manifold. We denote by $D_M$ the Nash bundle of
densities on $M$. It is the natural bundle whose smooth sections
are smooth measures. For the precise definition see e.g.
\cite{AG1}.
\end{notation}

An important property of Nash manifolds is
\begin{theorem}[Local triviality of Nash manifolds; \cite{Shi}, Theorem I.5.12  ] \label{loctriv}
Any Nash manifold can be covered by a finite number of open
submanifolds Nash diffeomorphic to $\R^n$.
\end{theorem}

\begin{definition}
Let $M$ be a Nash manifold. We denote by $\G(M):= \Sc^*(M,D_M)$
the \textbf{space of Schwartz generalized functions} on $M$.
Similarly, for a Nash bundle $E \to M$ we denote by $\G(M,E):=
\Sc^*(M,E^* \otimes D_M)$ the \textbf{space of Schwartz
generalized sections} of $E$.

In the same way, for any smooth manifold $M$ we denote by
$C^{-\infty}(M):= \cD(M,D_M)$ the \textbf{space of generalized
functions} on $M$ and for a smooth bundle $E \to M$ we denote by
$C^{-\infty}(M,E):= \cD(M,E^* \otimes D_M)$ the \textbf{space of
generalized sections} of $E$.
\end{definition}

Usual $L^1$-functions can be interpreted as Schwartz generalized
functions but not as Schwartz distributions. We will need several
properties of Schwartz functions from \cite{AG1}.

\begin{property}[\cite{AG1}, Theorem 4.1.3] \label{pClass}  $\Sc(\R ^n)$ = Classical
Schwartz functions on $\R ^n$.
\end{property}

\begin{property}[\cite{AG1}, Theorem 5.4.3] \label{pOpenSet}
Let $U \subset M$  be a (semi-algebraic) open subset, then
$$\Sc(U,E) \cong \{\phi \in \Sc(M,E)| \quad \phi \text{ is 0 on } M
\setminus U \text{ with all derivatives} \}.$$
\end{property}

\begin{property}[see \cite{AG1}, \S 5]\label{pCosheaf}
Let $M$ be a Nash manifold. Let $M = \bigcup U_i$ be a finite open
cover of $M$. Then a function $f$ on $M$ is a Schwartz function if
and only if it can be written as $f= \sum \limits _{i=1}^n f_i$
where $f_i \in \Sc(U_i)$ (extended by zero to $M$).

Moreover, there exists a smooth partition of unity $1 =\sum
\limits _{i=1}^n \lambda_i$ such that for any Schwartz function $f
\in \Sc(M)$ the function $\lambda_i f$ is a Schwartz function on
$U_i$ (extended by zero to $M$).
\end{property}

\begin{property}[see \cite{AG1}, \S 5]\label{pSheaf}
Let $M$ be a Nash manifold and $E$ be a Nash bundle over it. Let
$M = \bigcup U_i$ be a finite open cover of $M$. Let $\xi_i \in
\G(U_i,E)$ such that $\xi_i|_{U_j} = \xi_j|_{U_i}$. Then there
exists a unique $\xi \in \G(M,E)$ such that $\xi|_{U_i} = \xi_i$.
\end{property}

We will also use the following notation.
\begin{notation}
Let $M$ be a metric space and $x \in M$. We denote by $B(x,r)$ the
open ball with center $x$ and radius $r$.
\end{notation}
\subsection{Submersion principle}

\begin{theorem}[\cite{AG2}, Theorem 2.4.16] \label{SurSubSec}
Let $M$ and $N$ be Nash manifolds and $s:M \rightarrow N$ be a
surjective submersive Nash map. Then locally it has a Nash
section, i.e. there exists a finite open cover $N= \bigcup \limits
_{i=1}^k U_i$ such that $s$ has a Nash section on each $U_i$.
\end{theorem}

\begin{corollary} \label{EtLocIsNash}
An \et map $\phi:M \to N$ of Nash manifolds is locally an
isomorphism. That means that there exists a finite cover $M =
\bigcup U_i$ such that $\phi|_{U_i}$ is an isomorphism onto its
open image.
\end{corollary}

\begin{theorem} \label{NashEquivSub}
Let $p:M \to N$ be a Nash submersion of Nash manifolds. Then there
exist a finite  open (semi-algebraic) cover $M = \bigcup U_i$ and
isomorphisms $\phi_i:U_i \cong W_i$ and $\psi_i:p(U_i) \cong V_i$
where $W_i\subset \R^{d_i}$ and $V_i \subset \R^{k_i}$ are open
(semi-algebraic) subsets, $k_i \leq d_i$ and $p|_{U_i}$ correspond
to the standard projections.
\end{theorem}
\begin{proof}
The problem is local, hence without loss of generality we can
assume that $N = \R^k$, $M$ is an equidimensional closed
submanifold of $\R^n$ of dimension $d$, $d \geq k$, and $p$ is
given by the standard projection $\R^n \to \R^k$.

Let $\Omega$ be the set of all coordinate subspaces of $\R^n$ of
dimension $d$ which contain $N$. For any $V \in \Omega$ consider
the projection $pr:M \rightarrow V$. Define $U_V= \{x \in M | d_x
pr$ is an isomorphism $\}$. It is easy to see that $pr|_{U_V}$ is
\et and $\{U_V\}_{V \in \Omega}$ gives a finite cover of $M$. The
theorem now follows from the previous corollary (Corollary
\ref{EtLocIsNash}).
\end{proof}

\begin{theorem} \label{NashSub}
Let $\phi:M \to N$ be a Nash submersion of Nash manifolds. Let $E$
be a Nash bundle over $N$. Then\\
(i) there exists a unique continuous linear map
$\phi_*:\Sc(M,\phi^*(E)\otimes D_M) \to \Sc(N,E \otimes D_N)$ such
that for any $f \in \Sc(N,E^*)$ and $\mu \in \Sc(M,\phi^*(E)
\otimes D_M)$ we have $$\int_{x \in N} \langle f(x),\phi_*\mu(x)
\rangle = \int_{x \in M} \langle \phi^*f(x), \mu(x) \rangle.$$ In
particular, we mean that both integrals converge. \\
(ii) If $\phi$ is surjective then $\phi_*$ is surjective.
\end{theorem}
\begin{proof}
$ $\\ \indent (i)

Step 1. Proof for the case when $M= \R^n$, $N= \R^k$, $k \leq n$,
$\phi$ is the
standard projection and $E$ is trivial.\\
Fix Haar measure on $\R$ and identify $D_{\R^l}$ with the trivial
bundle for any $l$. Define $$\phi_*(f)(x):= \int _{y \in \R^{n-k}}
f(x,y)dy.$$ Convergence of the integral and the fact that
$\phi_*(f)$ is a Schwartz function follows from standard calculus.

Step 2. Proof for the case when  $M \subset \R^n$ and $N \subset
\R^k$ are open
(semi-algebraic) subsets, $\phi$ is the standard projection and $E$ is trivial.\\
Follows from the previous step and Property \ref{pOpenSet}.

Step 3. Proof for the case when  $E$ is trivial.\\
Follows from the previous step, Theorem \ref{NashEquivSub} and
partition of unity (Property \ref{pCosheaf}).

Step 4. Proof in the general case.\\
Follows from the previous step and partition of unity (Property
\ref{pCosheaf}).

(ii) The proof is the same as in (i) except of Step 2. Let us
prove (ii) in the case of Step 2. Again, fix Haar measure on $\R$
and identify $D_{\R^l}$ with the trivial bundle for any $l$. By
Theorem \ref{SurSubSec} and partition of unity (Property
\ref{pCosheaf}) we can assume that there exists a Nash section
$\nu:N \to M$. We can write $\nu$ in the form $\nu(x) = (x,s(x))$.

For any $x \in N$ define $R(x):= \sup\{r \in \R_{\geq 0}|
B(\nu(x),r) \subset M \}$. Clearly, $R$ is continuous and
positive. By Tarski - Seidenberg principle (see e.g. \cite[Theorem
2.2.3]{AG1}) it is semi-algebraic. Hence (by \cite[Lemma
A.2.1]{AG1}) there exists a positive Nash function $r(x)$ such
that $r(x) < R(x)$. Let $\rho \in \Sc(\R^{n-k})$ such that $\rho$
is supported in the unit ball and its integral is 1. Now let $f
\in \Sc(N)$. Let $g \in C^{\infty}(M)$ defined by $g(x,y):=
f(x)\rho((y-s(x))/r(x))/r(x)$ where $x \in N$ and $y \in
\R^{n-k}$. It is easy to see that $g \in \Sc(M)$ and $\phi_*g=f$.
\end{proof}

\begin{notation}
Let $\phi:M \to N$ be a Nash submersion of Nash manifolds. Let $E$
be a bundle on $N$. We denote by $\phi^*:\G(N,E) \to
\G(M,\phi^*(E))$ the dual map to $\phi_*$.
\end{notation}

\begin{remark}
Clearly, the map $\phi^*:\G(N,E) \to \G(M,\phi^*(E))$ extends to
the map $\phi^*:C^{-\infty}(N,E) \to C^{-\infty}(M,\phi^*(E))$
described in \cite[Theorem A.0.4]{AGS1}.
\end{remark}

\begin{proposition} \label{EnoughPull}
Let $\phi:M \to N$ be a surjective Nash submersion of Nash
manifolds. Let $E$ be a bundle on $N$. Let $\xi \in
C^{-\infty}(N)$. Suppose that $\phi^*(\xi) \in \G(M)$. Then $\xi
\in \G(N)$.
\end{proposition}
\begin{proof}
It follows from Theorem \ref{NashSub} and Banach Open Map Theorem
(see \cite[Theorem 2.11]{Rud}).
\end{proof}

\subsection{Frobenius reciprocity}
$ $\\
In this subsection we prove Frobenius reciprocity for Schwartz
functions on Nash manifolds.

\begin{proposition}
Let $M$ be a Nash manifold. Let $K$ be a Nash group. Let $E \to M$
be a Nash bundle. Consider the standard projection $p:K \times M
\to M$. Then the map $p^*:\G(M,E) \to \G(M \times K, p^*E)^K$ is
an isomorphism.
\end{proposition}
This proposition follows from in \cite[Proposition 4.0.11]{AG2}.

\begin{corollary}
Let a Nash group $K$ act on a Nash manifold $M$. Let $E$ be a
$K$-equivariant Nash bundle over $M$. Let $N \subset M$ be a Nash
submanifold such that the action map $K \times N \to M$ is
submersive. Then there exists a canonical map $$\HC: \G(M,E)^K \to
\G(N,E|_N).$$
\end{corollary}

\begin{theorem}
Let a Nash group $K$ act on a Nash manifold $M$. Let $N$ be a
$K$-transitive Nash manifold. Let $\phi:M \to N$ be a Nash
$K$-equivariant map.

Let $z \in N$ be a point and $M_z:= \phi^{-1}(z)$ be its fiber.
Let $K_z$ be the stabilizer of $z$ in $K$. Let $E$ be a
$K$-equivariant Nash vector bundle over $M$.

Then there exists a canonical isomorphism $$\Fr:
\G(M_z,E|_{M_z})^{K_z} \cong \G(M,\E)^K.$$
\end{theorem}
\begin{proof}

Consider the map $a_z:K \to N$ given by $a_z(g)=gz$. It is a
submersion. Hence by Theorem \ref{SurSubSec} there exists a finite
open cover $N= \bigcup \limits _{i=1}^k U_i$ such that $a_z$ has a
Nash section $s_i$ on each $U_i$. This gives an isomorphism
$\phi^{-1}(U_i) \cong U_i \times M_z$ which defines a projection
$p: \phi^{-1}(U_i) \to M_z$. Let $\xi \in \G(M_z,E|_{M_z})^{K_z}$.
Denote $\xi_i:=p^* \xi$. Clearly it does not depend on the section
$s_i$. Hence $\xi_i|_{U_i \cap U_j}=\xi_j|_{U_i \cap U_j}$ and
hence by Property \ref{pSheaf} there exists $\eta \in \G(M,\E)$
such that $\eta|_{U_i} = \xi_i$. Clearly $\eta$ does not depend on
the choices. Hence we can define $\Fr(\xi) = \eta$.

It is easy to see that the map $\HC: \G(M,E)^K \to
\G(M_{z},E|_{M_{z}})$ described in the last corollary gives the
inverse map.
\end{proof}

Since our construction coincides with the construction of
Frobenius reciprocity for smooth manifolds (see e.g. \cite[Theorem
A.0.3]{AGS1}) we obtain the following corollary.

\begin{corollary}
Part (ii) of Theorem \ref{Frob} holds.
\end{corollary}

\subsection{$K$-invariant distributions compactly supported modulo
$K$.} \label{KInvAreSchwartz}
$ $\\

In this subsection we prove Theorem \ref{CompSchwartz}. Let us
first remind its formulation.

\begin{theorem} \label{CompSchwartz2}
Let a Nash group $K$ act on a Nash manifold $M$. Let $E$ be a
$K$-equivariant Nash bundle over $M$. Let $\xi \in \cD(M,E)^K$
such that $\Supp(\xi)$ is Nashly compact modulo $K$. Then $\xi \in
\Sc^*(M,E)^K$.
\end{theorem}

For the proof we will need the following lemmas.

\begin{lemma} \label{RelComp}
Let $M$ be a Nash manifold. Let $C \subset M$ be a compact subset.
Then there exists a relatively compact open (semi-algebraic)
subset $U \subset M$ that includes $C$.
\end{lemma}

\begin{proof}
For any point $x \in C$ choose an affine chart, and let $U_x$ be
an open ball with center at $x$ inside this chart. Those $U_x$
give an open cover of $C$. Choose a finite subcover
$\{U_i\}_{i=1}^n$ and let $U:= \bigcup_{i=1}^n U_i$.
\end{proof}

\begin{lemma} \label{takayata}
Let $M$ be a Nash manifold. Let $E$ be a Nash bundle over $M$. Let
$U \subset M$ be a relatively compact open (semi-algebraic)
subset. Let $\xi \in \cD(M,E)$. Then $\xi|_U \in \Sc^*(U,E|_U)$.
\end{lemma}

\begin{proof}
It follows from the fact that extension by zero $ext:\Sc(U,E|_U)
\to C_c^{\infty}(M,E)$ is a continuous map.
\end{proof}

\begin{proof}[Proof of Theorem \ref{CompSchwartz2}]
Let $Z \subset M$ be a semi-algebraic closed subset and $C \subset
M$ be a compact subset such that $Supp(\xi) \subset Z \subset KC$.

Let $U \supset C$ be as in Lemma \ref{RelComp}. Let $\xi' :=
\xi|_{KU}$. Since $\xi|_{M - Z}=0$, it is enough to show that
$\xi'$ is Schwartz.

Consider the surjective submersion $m_U:K \times U \to KU$. Let
$$\xi'':=m_U^* (\xi') \in \cD(K \times U,m_U^*(E))^K.$$ By
Proposition \ref{EnoughPull}, it is enough to show that $$\xi''
\in \Sc^*(K \times U,m_U^*(E)).$$
By Frobenius reciprocity, $\xi''$ corresponds to $\eta \in
\cD(U,E)$. It is enough to prove that $\eta \in \Sc^*(U,E)$.
Consider the submersion $m:K \times M \to M$ and let $$\xi''' :=
m^*(\xi) \in \cD(K \times M,m^*(E)).$$ By Frobenius reciprocity,
$\xi'''$ corresponds to $\eta' \in \cD(M,E)$. Clearly $\eta =
\eta'|_U$. Hence by Lemma \ref{takayata}, $\eta \in \Sc^*(U,E)$.
\end{proof}
\section{Proof of the Archimedean Homogeneity Theorem} \label{AppRealHom}
The goal of this appendix is to prove Theorem \ref{ArchHom} for
Archimedean $F$. First we remind its formulation.

\setcounter{lemma}{0}

\begin{theorem} [Archimedean Homogeneity]
Let $V$ be a vector space over $F$. Let $B$ be a non-degenerate
symmetric bilinear form on $V$. Let $M$ be a Nash manifold. Let $L
\subset \Sc^*_{V\times M}(Z(B)\times M)$ be a non-zero subspace
such that for all $\xi \in L $ we have $\Fou_B(\xi) \in L$ and $B
\xi \in L$ (here $B$ is interpreted as a quadratic form).

Then there exists a non-zero distribution $\xi \in L$ which is
adapted to $B$.
\end{theorem}
\noindent Till the end of the section we assume that $F$ is
Archimedean and we fix $V$ and $B$.

First we will need some facts about the Weil representation. For a
survey on the Weil representation in the Archimedean case we refer
the reader to \cite[\S 1]{RS1}.

\begin{enumerate}
\item There exists a unique (infinitesimal\Rami{, $\R$-linear}) action
$\pi$ of $\sll_2(F)$ on $\Sc^*(V)$ such that\\
$\pi(\begin{pmatrix}
  0 & \Rami{z} \\
  0 & 0
\end{pmatrix}) \xi = -\mathrm{i} \pi Re(B)
\xi$ and $\pi(\begin{pmatrix}
  0 & 0 \\
  -\Rami{z} & 0
\end{pmatrix}) \xi = -\Fou_B^{-1} (\mathrm{i} \pi Re(\Rami{z}B)
\Fou_B(\xi))$.\\

\item It can be lifted to an action of the metaplectic group
$\Mp(2,F)$.

We will denote this action by $\Pi$.

\item In case $F=\C$ we have $\Mp(2,F)=\SL_2(F)$ and in case $F=\R$ the
group $\Mp(2,F)$ is a connected 2-fold covering of $\SL_2(F)$. We
will denote by $\eps \in \Mp(2,F)$ the central element of order 2
satisfying $\SL_2(F)=\Mp(2,F)/\{1, \eps\}.$

\item In case $F=\R$ we have $\Pi(\eps)=(-1)^{\dim V}$ and therefore if $\dim
V$ is even then $\Pi$ factors through $\SL_2(F)$ and if $\dim V$
is odd then no nontrivial subrepresentation of $\Pi$ factors
through $\SL_2(F)$. In particular if $\dim V$ is odd then $\Pi$
has no nontrivial finite-dimensional representations, since every
finite-dimensional representation of $\Mp(2,F)$ factors through
$\SL_2(F)$. \label{FinDimSubrep}

\item In case $F=\C$ or in case $\dim V$ is even we have $\Pi(\begin{pmatrix}
  t & 0 \\
  0 & t^{-1}
\end{pmatrix}) \xi=\delta^{-1}(t)|t|^{-\dim V/2} \rho(t) \xi$ and $\Pi(\begin{pmatrix}
  0 & 1 \\
  -1 & 0
\end{pmatrix}) \xi=
\gamma(B)^{-1} \Fou_B \xi.$
\end{enumerate}

We also need the following straightforward lemma.

\begin{lemma} \label{sl2rep}
\Rami{Assume that $F$ is Archimedean.}
Let $(\Lambda, L)$ be a continuous finite-dimensional
representation of $\SL_2(\Rami{F})$.  Then there exists a non-zero $\xi
\in L$ such that either $$\Lambda(\begin{pmatrix}
  t & 0 \\
  0 & t^{-1}
\end{pmatrix})\xi=\xi \text{ and } \Lambda(\begin{pmatrix}
  0 & 1 \\
  -1 & 0
\end{pmatrix})\xi \text{ is proportional to }\xi$$ or
$$\Lambda(\begin{pmatrix}
  t & 0 \\
  0 & t^{-1}
\end{pmatrix})\xi=t \xi,$$
\Rami{or $F=\C$ and
$$\Lambda(\begin{pmatrix}
  t & 0 \\
  0 & t^{-1}
\end{pmatrix})\xi=\bar t \xi,$$
or $F=\C$ and
$$\Lambda(\begin{pmatrix}
  t & 0 \\
  0 & t^{-1}
\end{pmatrix})\xi=\frac{\bar t}{t} \xi$$
}

for all $t$.
\end{lemma}

Now we are ready to prove the theorem.
\begin{proof}[Proof of Theorem \ref{ArchHom}]
Without loss of generality assume $M=pt$.

Let $\xi \in L$ be a non-zero distribution. Let $L':=
U_{\C}(\sll_2(\Rami{F})) \xi \subset L$. Here, $U_{\C}$ means the
complexified universal enveloping algebra.

We are given that $\xi, \Fou_B(\xi) \in \Sc^*_V(Z(B))$. By Lemma
\ref{FinDim} below this implies that $L' \subset \Sc^*(V)$ is
finite-dimensional. Clearly, $L'$ is also a subrepresentation of
$\Pi$. Therefore by Fact (\ref{FinDimSubrep}), $F=\C$ or $\dim V$
is even. Hence $\Pi$ factors through $\SL_2(F)$.

Now by Lemma \ref{sl2rep} there exists $\xi' \in L'$ which is
$B$-adapted.
\end{proof}

\begin{lemma}\label{FinDim}
Let $V$ be a representation of $\sll_2$. Let $v \in V$ be a vector
such that $e^k v=f^n v=0$ for some $n,k$. Then the representation
generated by $v$ is finite-dimensional.\footnote{For our purposes
it is enough to prove this lemma for k=1.}
\end{lemma}
This lemma is probably well-known. Since we have not found any
reference we include the proof.
\begin{proof}
The proof is by induction on k.\\\\
Base k=1:\\
It is easy to see that $$e^l f^l v=l!(\prod_{i=0}^{l-1}(h-i))v$$
for all $l$. This can be checked by direct computation, and also
follows from the fact that $e^l f^l$ is of weight $0$, hence it
acts on the singular vector $v$ by its Harish-Chandra projection
which is $$\HC(e^l f^l)=l! \prod_{i=0}^{l-1}(h-i).$$

Therefore $(\prod_{i=0}^{n-1}(h-i))v=0$.

Hence $W:=U_{\C}(h) v$ is finite-dimensional and $h$ acts on it
semi-simply. Here, $U_{\C}(h)$ denotes the universal enveloping
algebra of $h$.  Let $\{v_i\}_{i=1}^m$ be an eigenbasis of $h$ in
$W$. It is enough to show that $U_{\C}(\sll_2)v_i$ is
finite-dimensional for any $i$. Note that $e|_W=f^n|_W=0$. Now,
$U_{\C}(\sll_2)v_i$ is finite-dimensional by the
Poincare-Birkhoff-Witt
Theorem.\\\\
Induction step:\\
Let $w:=e^{k-1}v$. Let us show that $f^{n+k-1} w=0$. Consider the
element $f^{n+k-1}e^{k-1}\in U_{\C}(\sll_2)$. It is of weight
$-2n$, hence by the Poincare-Birkhoff-Witt Theorem it can be
rewritten as a combination of elements of the form $e^a h^b f^c$
such that $c-a= n$ and hence $c \geq n$. Therefore
$f^{n+k-1}e^{k-1} v=0$.

Now let $V_1:=U_{\C}(\sll_2) v$ and $V_2:=U_{\C}(\sll_2) w$. By
the base of the induction $V_2$ is finite-dimensional, by the
induction hypotheses  $V_1/V_2$ is finite-dimensional, hence $V_1$
is finite-dimensional.
\end{proof}

\section{Localization Principle}\label{SecRedLocPrin}
$ \quad \quad \quad \quad  \quad \quad$ by Avraham Aizenbud,
Dmitry Gourevitch
and Eitan Sayag\\\\
\setcounter{lemma}{0}
In this appendix we formulate and prove the Localization Principle
in the case of a reductive group $G$ acting on a smooth affine
variety $X$. This is of interest only for Archimedean $F$ since for
$l$-spaces, a more general version of this principle has been
proven in \cite{Ber}.
In \cite{AGS2}, we formulated without proof a Localization Principle in the
setting of differential geometry. Admittedly, we currently do not have a proof
of this principle in such a general setting.
However, the current generality is sufficiently wide for all applications we encountered up
to now, including the one in \cite{AGS2}.


\begin{theorem}[Localization Principle] \label{LocPrin2}
Let a reductive group $G$ act on a smooth algebraic variety $X$.
Let $Y$ be an algebraic variety and $\phi:X \to Y$ be an affine
algebraic $G$-invariant map. Let $\chi$ be a character of $G(F)$.
Suppose that for any $y \in Y(F)$ we have
$\cD_{X(F)}((\phi^{-1}(y))(F))^{G(F),\chi}=0$. Then
$\cD(X(F))^{G(F),\chi}=0$.
\end{theorem}

\begin{proof}
Clearly, it is enough to prove the theorem for the case when $X$
is affine, $Y = X/G$ and $\phi = \pi_X(F)$. By the Generalized
Harish-Chandra Descent (Corollary \ref{Strong_HC_Cor}), it is
enough to prove that for any $G$-semisimple $x\in X(F)$, we have
$$\cD_{N_{Gx,x}^X(F)}(\Gamma(N_{Gx,x}^X))^{G_x(F),\chi} = 0.$$

Let $(U,p,\psi,S,N)$ be an analytic Luna slice at $x$. Clearly,
$$\cD_{N_{Gx,x}^X(F)}(\Gamma(N_{Gx,x}^X))^{G_x(F),\chi} \cong
\cD_{\psi(S)}(\Gamma(N_{Gx,x}^X))^{G_x(F),\chi}\cong
\cD_{S}(\psi^{-1}(\Gamma(N_{Gx,x}^X)))^{G_x(F),\chi}.$$ By
Frobenius reciprocity (Theorem \ref{Frob}),
$$\cD_{S}(\psi^{-1}(\Gamma(N_{Gx,x}^X)))^{G_x(F),\chi}=\cD_{U}(G(F)\psi^{-1}(\Gamma(N_{Gx,x}^X)))^{G(F),\chi}.$$

By Lemma \ref{Gamma}, $$G(F)\psi^{-1}(\Gamma(N_{Gx,x}^X))= \{y \in
X(F)| x \in \overline{G(F)y} \}.$$
Hence by Corollary \ref{EquivClassClosed},
$G(F)\psi^{-1}(\Gamma(N_{Gx,x}^X))$ is closed in $X(F)$. Hence
$$\cD_{U}(G(F)\psi^{-1}(\Gamma(N_{Gx,x}^X)))^{G(F),\chi}
 =
\cD_{X(F)}(G(F)\psi^{-1}(\Gamma(N_{Gx,x}^X)))^{G(F),\chi}.$$ Now,
$$G(F)\psi^{-1}(\Gamma(N_{Gx,x}^X)) \subset
\pi_X(F)^{-1}(\pi_X(F)(x))$$ and we are given
$$\cD_{X(F)}(\pi_X(F)^{-1}(\pi_X(F)(x)))^{G(F),\chi}=0$$ for any
$G$-semisimple $x$.
\end{proof}

\begin{remark} \label{LocPrinS}
An analogous statement holds for Schwartz distributions and the
proof is the same.
\end{remark}

\begin{corollary} \label{LocPrinSub}
Let a reductive group $G$ act on a smooth algebraic variety $X$.
Let $Y$ be an algebraic variety and $\phi:X \to Y$ be an affine
algebraic $G$-invariant submersion. Suppose that for any $y \in
Y(F)$ we have $\Sc^*(\phi^{-1}(y))^{G(F),\chi}=0$. Then
$\cD(X(F))^{G(F),\chi}=0$.
\end{corollary}

\begin{proof}
%
For any $y \in Y(F)$, denote $X(F)_y:=(\phi^{-1}(y))(F)$. Since
$\phi$ is a submersion, for any $y \in Y(F)$ the set $X(F)_y$ is a
smooth manifold. Moreover, $d\phi$ defines an isomorphism between
$N_{X(F)_y,z}^{X(F)}$ and $T_{Y(F),y}$ for any $z \in X(F)_y$.
Hence the bundle $CN_{X(F)_y}^{X(F)}$ is a trivial
$G(F)$-equivariant bundle.

We know that $$\Sc^*(X(F)_y)^{G(F),\chi}=0.$$ Therefore for any
$k$, we have
$$\Sc^*(X(F)_y,\Sym^k(CN_{X(F)_y}^{X(F)}))^{G(F),\chi}=0.$$ Thus by
Theorem \ref{NashFilt}, $\Sc^*_{X(F)}(X(F)_y)^{G(F),\chi}=0$. Now,
by Theorem \ref{LocPrin2} (and Remark \ref{LocPrinS}) this implies
that $\Sc^*(X(F))^{G(F),\chi}=0$. Finally, by Theorem
\ref{NoSNoDist} this implies $\cD(X(F))^{G(F),\chi}=0$.
\end{proof}

\begin{remark} \label{RemLocVectSys}
Theorem \ref{LocPrin} and Corollary \ref{LocPrinSub} admit obvious
generalizations to constant vector systems. The same proofs hold.
\end{remark}

\newpage

\section{Diagram} \label{Diag}
The following diagram illustrates the interrelations of the
various
properties of a symmetric pair $(G,H)$. On the non-trivial implications we put the numbers of the statements that prove them.
Near the important notions we put the numbers of the definitions which define those notions. \\
\small
\newlength{\gnat}

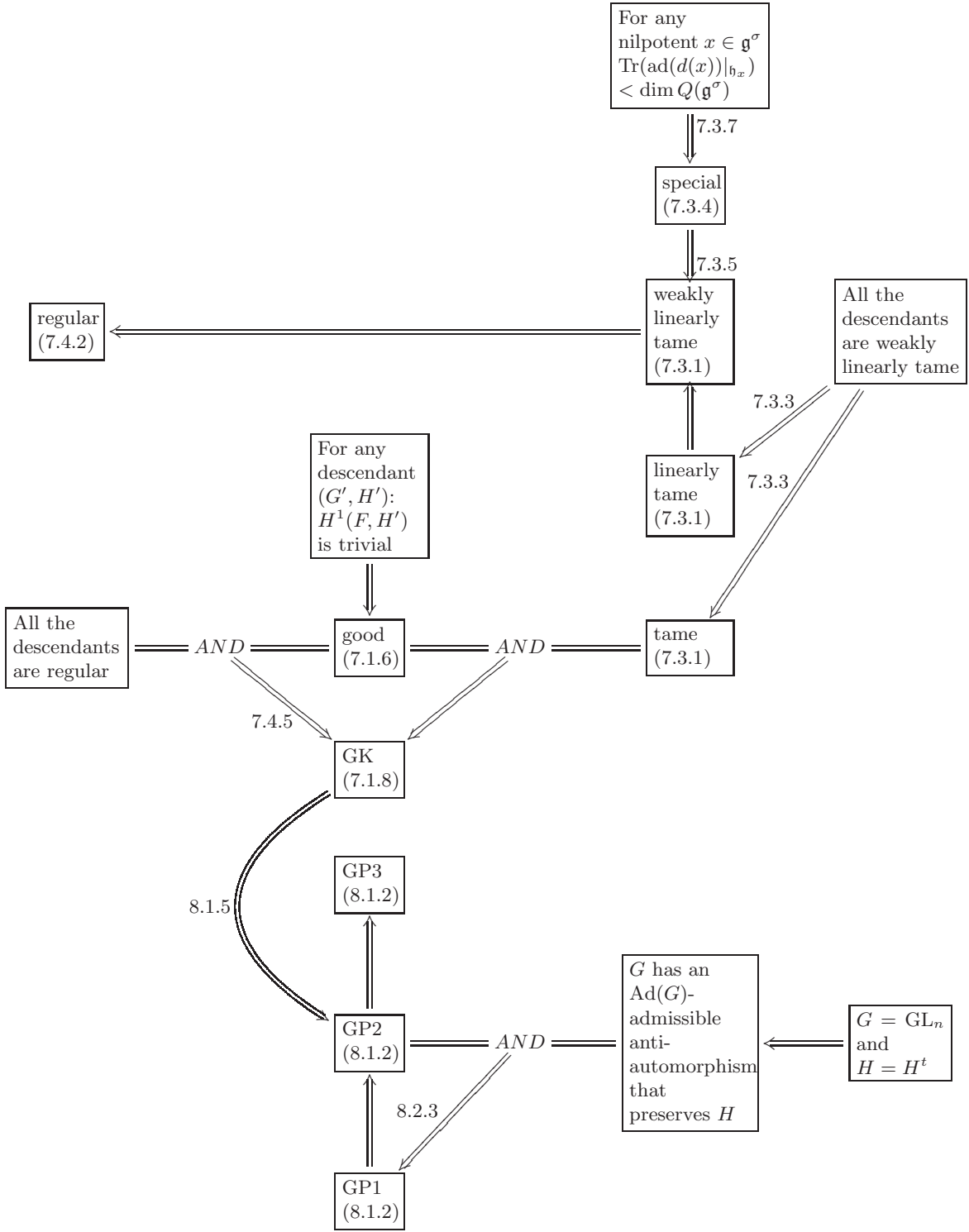
\setlength{\gnat}{10pt} $  \xymatrix{& & & &
\framebox{\parbox{65pt}{For any \\ nilpotent $x \in \gd$
$\tr(\ad(d(x))|_{\h_x})$ $< \dim Q(\gd) $}}\ar@{=>}[d]^{\ref{SpecCrit}} & &\\
& & & &\framebox{\parbox{25pt}{special\\ (\ref{DefSpecPair})}}\ar@{=>}[d]^{\ref{SpecWeakReg}}& & \\
\framebox{\parbox{27pt}{regular \\ (\ref{DefReg})}}& & &
&\framebox{\parbox{33pt}{weakly linearly tame
\\ (\ref{DefTamePairs})}}\ar@{=>}[llll]&\framebox{\parbox{53pt}{All the\\ descendants are
weakly \\ linearly tame}}
\ar@{=>}[dl]_{\ref{LinDes}}\ar@{=>}[ddl]_{\ref{LinDes}}\\
& &\framebox{\parbox{47pt}{For any\\ descendant $(G',H')$:\\
$H^1(F,H')$ is trivial }}
\ar@{=>}[d] & &\framebox{\parbox{33pt}{linearly tame \\ (\ref{DefTamePairs})}}\ar@{=>}[u]&\\
 \framebox{\parbox{49pt}{All the\\ descendants are regular}}\ar@{=}[r]&AND\ar@{=}[r]\ar@{=>}
 [rd]_{\ref{GoodHerRegGK}}&\framebox{\parbox{25pt}{good \\ (\ref{DefGoodPair})}}\ar@{=}[r]&AND\ar@{=}[r]\ar@{=>}[ld] &\framebox{\parbox{33pt}{tame \\ (\ref{DefTamePairs})}}
 &\\
&&\framebox{\parbox{25pt}{GK \\ (\ref{DefGKPair})}}\ar@/_5pc/@{=>}[dd]_{\ref{DistCrit}}&&&&\\
&&\framebox{\parbox{25pt}{GP3 \\ (\ref{GPs})}}&&&\\
&&\framebox{\parbox{25pt}{GP2
\\ (\ref{GPs})}}\ar@{=>}[u]\ar@{=}[r]&AND\ar@{=>}[ld]_{\ref{GKCor}}\ar@{=}[r]&\framebox{\parbox{55pt}
{$G$ has an\\ $\Ad(G)$-admissible anti-automorphism that\\
preserves $H$}}& \ar@{=>}[l] \framebox{\parbox{40pt}
{$G=\GL_n$ and\\ $H=H^t$ }} \\
&&\framebox{\parbox{25pt}{GP1 \\ (\ref{GPs})}}\ar@{=>}[u]&&&\\
 }$


\newpage

\begin{thebibliography}{99}
\bibitem[\href{http://imrn.oxfordjournals.org/cgi/reprint/2008/rnm155/rnm155?ijkey=bddq0itkXKrVjlG&keytype=ref}{AG08a}]{AG1} A. Aizenbud,  D.
Gourevitch, {\it Schwartz functions on Nash Manifolds,}
International Mathematics Research Notices, Vol. 2008, n.5,
Article ID rnm155, 37 pages. DOI: 10.1093/imrn/rnm155. See also
arXiv:0704.2891 [math.AG].
\bibitem[\href{http://arxiv.org/PS_cache/arxiv/pdf/0802/0802.3305v2.pdf}{AG08b}]{AG2} Aizenbud, A.; Gourevitch, D.: {\it De-Rham theorem and Shapiro lemma for Schwartz functions on Nash manifolds.}
To appear in the Israel Journal of Mathematics. See also
arXiv:0802.3305v2 [math.AG].
\bibitem[\href{http://arxiv.org/abs/0805.2504}{AG08c}]{AG_RegSymPairs} Aizenbud, A.; Gourevitch, D.:
{\it Some regular symmetric pairs.} To appear in Transactions of
AMS. See also arxiv:0805.2504[math.RT].




\bibitem[\href{http://arxiv.org/pdf/0709.4215v1}{AGRS07}]{AGRS} A. Aizenbud,  D. Gourevitch, S. Rallis, G. Schiffmann, {\it
Multiplicity One Theorems}, arXiv:0709.4215v1 [math.RT], To appear
in the Annals of Mathematics.
\bibitem[\href{http://arxiv.org/pdf/0709.1273v4}{AGS08}]{AGS1} A. Aizenbud,  D. Gourevitch, E. Sayag : {\it
$(GL_{n+1}(F),GL_n(F))$ is a Gelfand pair for any local field
$F$}, postprint: arXiv:0709.1273v4[math.RT]. Originally published
in: Compositio Mathematica, \textbf{144} , pp 1504-1524 (2008),
doi:10.1112/S0010437X08003746.

\bibitem[\href{http://www.springerlink.com/content/48436n62526244m3/}{AGS09}]{AGS2} A. Aizenbud,  D. Gourevitch, E. Sayag : {\it
$(O(V \oplus F), O(V))$  is a Gelfand pair for any quadratic space
$V$ over a local field $F$,} Mathematische Zeitschrift,
\textbf{261} pp 239-244 (2009), DOI: 10.1007/s00209-008-0318-5.
See also arXiv:0711.1471[math.RT].

\bibitem[\href{http://arxiv.org/abs/0811.2768}{Aiz08}]{Aiz}A. Aizenbud, {\it A partial analog of integrability theorem for
distributions on p-adic spaces and applications.}
arXiv:0811.2768[math.RT].

\bibitem[\href{http://arxiv.org/abs/0810.1853v1}{AS08}]{AS} A. Aizenbud, E. Sayag
{\it Invariant distributions on non-distinguished nilpotent orbits
with application to the Gelfand property of (GL(2n,R),Sp(2n,R))},
arXiv:0810.1853 [math.RT].



\bibitem[\href{http://annals.math.princeton.edu/issues/2003/Baruch.pdf}{Bar03}]{Bar} E.M.
Baruch, {\it A proof of Kirillov's conjecture,} Annals of
Mathematics, \textbf{158}, 207-252 (2003).
\bibitem[BCR98]{BCR} Bochnak, J.; Coste, M.; Roy, M-F.:
{\it Real Algebraic Geometry} Berlin: Springer, 1998.

\bibitem[Ber84]{Ber}
Joseph~N. Bernstein, \emph{{$P$}-invariant distributions on {${\rm
GL}(N)$} and
  the classification of unitary representations of {${\rm GL}(N)$}
  (non-{A}rchimedean case)}, Lie group representations, II (College Park, Md.,
  1982/1983), Lecture Notes in Math., vol. 1041, Springer, Berlin, 1984,
  pp.~50--102. \MR{MR748505 (86b:22028)}

\bibitem[\href{http://www.jstor.org/stable/1970884}{Brk71}]{Brk}
D. Birkes, Orbits of linear algebraic groups, Ann. of Math. (2)93
(1971), 459-475.


\bibitem[BvD94]{Bos-vD} E. P. H. Bosman and G. Van Dijk, {\it A New Class of Gelfand Pairs,}
Geometriae Dedicata \textbf{50}, 261-282, 261 @ 1994
KluwerAcademic Publishers. Printed in the Netherlands (1994).

\bibitem[BZ76]{BZ}
I.~N. Bern{\v{s}}te{\u\i}n and A.~V. Zelevinski{\u\i},
\emph{Representations of
  the group {$GL(n,F),$} where {$F$} is a local non-{A}rchimedean field},
  Uspehi Mat. Nauk \textbf{31} (1976), no.~3(189), 5--70. \MR{MR0425030 (54
  \#12988)}

\bibitem[\href{http://citeseer.ist.psu.edu/349964.html}{Dre00}]{Dre}
J.M.Drezet, {\it Luna's Slice Theorem And Applications}, 23d
Autumn School in Algebraic Geometry "Algebraic group actions and
quotients" Wykno (Poland), (2000).
\bibitem[Fli91]{Fli} Y.Z. Flicker: {\it On distinguished representations}, J. Reine Angew.
Math. \textbf{418} (1991), 139-172.

\bibitem[Gel76]{Gel} S. Gelbart: {\it Weil's Representation and the
Spectrum of the Metaplectic Group}, Lecture Notes in Math.,
\textbf{530}, Springer, Berlin-New York (1976).

\bibitem[GK75]{GK}
I.~M. Gelfand and D.~A. Kajdan, \emph{Representations of the group
  {${\rm GL}(n,K)$} where {$K$} is a local field}, Lie groups and their
  representations (Proc. Summer School, Bolyai J\'anos Math. Soc., Budapest,
  1971), Halsted, New York, 1975, pp.~95--118. \MR{MR0404534 (53 \#8334)}

\bibitem[Gro91]{Gross} B. Gross,
{\it Some applications Gelfand pairs to number theory,} Bull.
Amer. Math. Soc. (N.S.) \textbf{24}, no. 2, 277--301 (1991).
%
\bibitem[HC99]{HCh}
Harish-Chandra, \emph{Admissible invariant distributions on
reductive
  {$p$}-adic groups}, University Lecture Series, vol.~16, American Mathematical
  Society, Providence, RI, 1999, Preface and notes by Stephen DeBacker and Paul
  J. Sally, Jr. \MR{MR1702257 (2001b:22015)}

\bibitem[Jac62]{Jac} N. Jacobson, {\it Lie Algebras},
Interscience Tracts on Pure and Applied Mathematics, no.
\textbf{10}. Interscience Publishers, New York (1962).
\bibitem[JR96]{JR}
Herv{\'e} Jacquet and Stephen Rallis, \emph{Uniqueness of linear
periods},
  Compositio Math. \textbf{102} (1996), no.~1, 65--123. \MR{MR1394521
  (97k:22025)}

\bibitem[\href{http://www.jstor.org/view/00029327/di994396/99p0264d/0}{KR73}]{KR}
B. Kostant and S. Rallis, Orbits and representations associated
with symmetric spaces, Amer. J. Math. 93 (1971), 753-809.
\bibitem[\href{http://www.numdam.org/numdam-bin/fitem?id=MSMF_1973__33__81_0}{Lun73}]{Lun}  D. Luna,
 {\it Slices \'{e}tales.}, M\'{e}moires de la Soci\'{e}t\'{e} Math\'{e}matique de France, \textbf{33} (1973), p. 81-105

\bibitem[\href{http://www.jstor.org/stable/2373666}{Lun75}]{Lun2} D. Luna,
{\it Sur certaines operations differentiables des groupes de Lie},
Amer. J.Math. \textbf{97} (1975), 172-181.

\bibitem[Mum99]{Mum} D. Mumford, {\it The Red Book of Varieties and Schemes}, second edition, Lecture
Notes in Math., vol. \textbf{1358}, Springer- Verlag, New York,
(1999).
\bibitem[Pra90]{Prasad}
Dipendra Prasad, \emph{Trilinear forms for representations of
{${\rm GL}(2)$}
  and local {$\epsilon$}-factors}, Compositio Math. \textbf{75} (1990), no.~1,
  1--46. \MR{MR1059954 (91i:22023)}

\bibitem[\href{http://www.jstor.org/view/00029327/di994428/99p0186g/0}{RS78}]{RS1} S. Rallis,  G. Schiffmann,
{\it Automorphic Forms Constructed from the Weil Representation:
Holomorphic Case}, American Journal of Mathematics, Vol.
\textbf{100}, No. 5, pp. 1049-1122 (1978).
\bibitem[\href{http://arxiv.org/abs/0705.2168v1}{RS07}]{RS2} S. Rallis,  G. Schiffmann, {\it
Multiplicity one Conjectures}, arXiv:0705.2168v1 [math.RT].
\bibitem[\href{http://muse.jhu.edu/journals/american_journal_of_mathematics/v118/118.1rader.pdf}{RR96}]{RR} C. Rader and S. Rallis :{\it Spherical Characters On p-Adic Symmetric Spaces}, American Journal of Mathematics \textbf{118} (1996), 91-178.
\bibitem[Rud73]{Rud} W. Rudin :{\it Functional analysis} New York : McGraw-Hill, 1973.

\bibitem[\href{http://arxiv.org/abs/0805.2625}{Say08a}]{Say1} E.
Sayag, {\it (GL(2n,C),SP(2n,C)) is a Gelfand Pair}.
arXiv:0805.2625v1 [math.RT].
\bibitem[Say08b]{Say2} E. Sayag, {\it Regularity of invariant distributions on nice symmetric spaces and Gelfand property of symmetric pairs}, preprint.
\bibitem[Ser64]{Ser} J.P. Serre: {\it Lie Algebras and Lie Groups} Lecture Notes in Mathematics \textbf{1500}, Springer-Verlag, New York, (1964).
\bibitem[Sha74]{Shalika} J. Shalika,
{\it The multiplicity one theorem for $GL_{n}$}, Annals of
Mathematics , \textbf{100}, N.2,  171-193 (1974).
\bibitem[Shi87]{Shi} M. Shiota, {\it Nash Manifolds},  Lecture Notes in Mathematics \textbf{1269}
(1987).
\bibitem[Tho84]{Thomas} E.G.F. Thomas,
{\it The theorem of Bochner-Schwartz-Godement for generalized
Gelfand pairs}, Functional Analysis: Surveys and results III,
Bierstedt, K.D., Fuchsteiner, B. (eds.), Elsevier Science
Publishers B.V. (North Holland), (1984).
\bibitem[vD86]{vD} G. van Dijk, {\it On a class of generalized Gelfand
pairs}, Math. Z. \textbf{193}, 581-593 (1986).
\bibitem[vDP90]{vD-P} van Dijk, M. Poel,  {\it The irreducible unitary $GL_{n-1}(\R)$-spherical
representations of $SL_{n}(\R)$}. Compositio Mathematica, 73 no. 1
(1990), p. 1-307.
\bibitem[Wal88]{Wal1} N. Wallach,
{\it Real Reductive groups I}, Pure and Applied Math.
\textbf{132}, Academic Press, Boston, MA (1988).
\bibitem[Wal92]{Wal2} N. Wallach,
{\it Real Reductive groups II} , Pure and Applied Math.
\textbf{132-II}, Academic Press, Boston, MA (1992).
\bibitem[\href{http://bib.math.uni-bonn.de/pdf/BMS-374.pdf}{Yak04}]{Yak} O. Yakimova,
{\it Gelfand pairs}, PhD thesis submitted to Bonn university
(2004).\\ Availiable at
http://bib.math.uni-bonn.de/pdf/BMS-374.pdf
\end{thebibliography}
\end{document}